\documentclass{amsart}

\copyrightinfo{2011}{Nicholas Coxon}

\newtheorem{theorem}{Theorem}[section]
\newtheorem{lemma}[theorem]{Lemma}
\newtheorem{corollary}[theorem]{Corollary}

\theoremstyle{definition}
\newtheorem{definition}[theorem]{Definition}
\newtheorem{example}[theorem]{Example}
\newtheorem{algorithm}[theorem]{Algorithm}

\theoremstyle{remark}
\newtheorem{remark}[theorem]{Remark}

\numberwithin{equation}{section}

\def\vec#1{\mathchoice{\mbox{\boldmath$\displaystyle#1$}}
{\mbox{\boldmath$\textstyle#1$}}
{\mbox{\boldmath$\scriptstyle#1$}}
{\mbox{\boldmath$\scriptscriptstyle#1$}}}
\newcommand{\seq}[3]{\ensuremath{#1_{#2},\ldots,#1_{#3}}}
\newcommand{\norm}[1]{\left\|#1\right\|}
\newcommand{\Q}{\mathbb{Q}}
\newcommand{\Z}{\mathbb{Z}}
\newcommand{\C}{\mathbb{C}}
\newcommand{\R}{\mathbb{R}}
\newcommand{\res}{\mathrm{Res}}
\def\revdots{\mathinner{\mkern1mu\raise1pt\vbox{\kern7pt\hbox{.}}\mkern2mu\raise4pt\hbox{.}\mkern2mu \raise7pt\hbox{.}\mkern1mu}}

\begin{document}
\title{On Nonlinear Polynomial Selection for the Number Field Sieve}
\author{Nicholas Coxon}
\address{School of Mathematics and Physics, The University of Queensland, Brisbane, QLD 4072, Australia}
\curraddr{}
\email{ncoxon@maths.uq.edu.au}
\thanks{}
\date{June 28, 2013}
\keywords{Integer factorisation, number field sieve, polynomial selection}
\begin{abstract} Nonlinear polynomial selection algorithms for the number field sieve address the problem of constructing polynomials with small coefficients by reducing to instances of the well-studied problem of finding short vectors in lattices.  The reduction rests upon the construction of modular geometric progressions with small terms.  In this paper, the methods used to construct the geometric progressions are extended, resulting in the development of two nonlinear polynomial selection algorithms.
\end{abstract}
\maketitle
%

\section{Introduction}\label{sec:intro}

To factor an integer $N$, the number field sieve~\cite{lenstra93} begins with the selection of two low-degree coprime irreducible polynomials $f_{1},f_{2}\in\Z[x]$ with a common root modulo $N$.  If $F_{i}\in\Z[x,y]$ is the homogenisation of $f_{i}$, for $i=1,2$, then the time taken to factor $N$ depends on the supply of coprime integer pairs $(a,b)$ for which $F_{1}(a,b)$ and $F_{2}(a,b)$ are free of prime factors greater than a preselected bound.  Pairs with this property, called \emph{relations}, are identified by sieving.  The polynomial selection problem is to determine a choice of polynomials that help minimise the time taken by the sieve stage of the number field sieve.

The size of the values taken by the polynomials $F_{1}$ and $F_{2}$ is a key factor in determining the supply of relations \cite{murphy98,murphy99}.  Polynomial selection algorithms address this factor by endeavouring to generate polynomials with small coefficients.  The efforts of research into this problem are divided between two different approaches: so-called linear and nonlinear algorithms.  Linear algorithms were introduced during the development of the number field sieve~\cite{buhler93} and subsequently improved by Montgomery and Murphy~\cite{murphy99}, and Kleinjung~\cite{kleinjung06,kleinjung08}.  They have been used in a series of record-setting factorisations, culminating in the factorisation of a $768$-bit RSA modulus~\cite{kleinjung10}.  The designation as ``linear" is on account of the algorithms producing polynomial pairs such that one polynomial is linear.  This property results in a disparity in their degrees, and thus a disparity in the size of the values $F_{1}(a,b)$ and $F_{2}(a,b)$, reducing their yield of relations \cite[Section~6.2.7]{crandall05}.  In contrast to linear algorithms, nonlinear algorithms produce pairs of nonlinear polynomials with equal or almost equal degrees.  Nevertheless, they have received little practical attention since until recently nonlinear algorithms were only capable of producing pairs of quadratic polynomials, which limited the range for which they were competitive with linear algorithms to numbers of at most 110--120 digits \cite[Section~2.3.1]{murphy99}.  Consequently, the development of nonlinear algorithms has fallen behind that of linear algorithms.

Nonlinear polynomial selection algorithms for the number field sieve use geometric progressions with small terms modulo $N$ to generate pairs of nonlinear polynomials.  Montgomery introduced the approach with the two quadratics algorithm (see \cite[Section~5]{elkenbracht96} and \cite[Section~2.3.1]{murphy99}), which produces pairs of quadratic polynomials with coefficients of optimal size.  In addition, Montgomery~\cite{montgomery93,montgomery06} outlined a generalisation of the quadratic algorithm to higher degrees.  However, the problem of how to construct the geometric progressions required by the generalisation remains open for higher degrees.  Recent developments in geometric progression construction and relaxations of the requirements of Montgomery's approach have lead to a succession of new nonlinear algorithms~\cite{williams10,prest10,koo11}.  Building upon these developments, this paper presents two nonlinear polynomial generation algorithms.

The paper is organised as follows. In the next section, polynomial coefficient norms used throughout the paper are introduced, and generalities on real lattices are discussed.  Section~\ref{sec:gp-review} recalls the overall approach of nonlinear polynomial generation, introduces some properties of orthogonal lattices, then reviews nonlinear generation and existing algorithms in detail.  Finally, new nonlinear generation algorithms are presented and analysed in Section~\ref{sec:gp-new} and Section~\ref{sec:gp-new-mod}.

\section{Preliminaries}

\subsection{Skewed coefficient norms and the resultant bound}\label{sec:skewed-norms}

Coefficient norms used to measure the coefficient size of number field sieve polynomials are introduced.  Then a lower bound on the coefficient size of polynomial pairs with a common root modulo $N$ is derived.

\subsubsection{Skewed coefficient norms}

After polynomial selection, the number field sieve uses sieving to identify relations in some region $\mathcal{A}\subset\R^{2}$.  The form of sieve region is determined by the methods of sieving used, but may be approximated by a rectangular region of the form $\mathcal{A}=[-A,A]\times[0,B]$.  The area of the sieve region is approximately determined by the size of the input $N$.  Therefore, by assuming that the region's area is fixed, it follows that a rectangular sieve region $\mathcal{A}$ is determined by the parameter $s=A/B$, called the \emph{skew} of the region.   

Given two polynomials $f_{1}$ and $f_{2}$, the size properties of their respective homogenisations $F_{1}$ and $F_{2}$ over $\mathcal{A}$ are quantified by the integral
\begin{equation*}\label{eqn:int-over-A}
  \int_{\mathcal{A}}\left|F_{1}(x,y)F_{2}(x,y)\right|\,dxdy%
  =\left(AB\right)^{\frac{\deg f_{1}+\deg f_{2}+2}{2}}\cdot\int^{1}_{0}\int^{1}_{-1}\prod^{2}_{i=1}\left|F_{i}\left(x\sqrt{s},y/\sqrt{s}\right)\right|\,dxdy.
\end{equation*}
Polynomials that minimise this integral are expected to yield more relations in $\mathcal{A}$ than others.  The integrand on the right motivates the following choice of norm:
\begin{definition}  Let $f=\sum^{d}_{i=0}a_{i}x^{i}\in\R[x]$ be a degree $d$ polynomial.  For a given skew $s>0$, the \emph{skewed $2$-norm} of $f$ is defined to be
\begin{equation*}
  \norm{f}_{2,s}=\left(\sum^{d}_{i=0}\left|a_{i}s^{i-\frac{d}{2}}\right|^{2}\right)^{\frac{1}{2}}.
\end{equation*}
If $s=1$, then $\norm{f}_{2,s}$ is the \emph{$2$-norm} of $f$, which is denoted $\norm{f}_{2}$.
\end{definition}

A \emph{skew} of a polynomial $f\in\R[x]$ is any value $s>0$ for which $\norm{f}_{2,s}$ is minimal.  A number field sieve polynomial $f$ has good \emph{root properties} if the homogenised polynomial $F(x,y)$ has many roots modulo small primes, where a root $(r_{1},r_{2})$ modulo $p$ is considered as a point $(r_{1}:r_{2})$ on the projective line $\mathbf{P}^{1}(\mathbb{F}_{p})$.  Boender, Brent, Montgomery and Murphy~\cite{boender97,murphy98a,murphy98,murphy99} provide heuristic evidence which suggests good root properties increase a polynomial's yield.  The method of improving root properties by rotation~\cite{murphy99,gower03,bai11b} is more effective for polynomials with large skew, since large rotations may be used without impinging on their size properties.  Additionally, highly skewed polynomials are often better suited to sieving over a region with large skew, which may permit sieving to be performed more efficiently~\cite{kleinjung08}.  Thus, polynomial selection algorithms aim to generate highly skewed polynomials.

\subsubsection{The resultant bound}

For nonzero coprime polynomials $f_{1},f_{2}\in\Z[x]$ with a common root modulo $N$, the \emph{resultant bound} provides a lower bound on the $2$-norms of $f_{1}$ and $f_{2}$:
\begin{equation}\label{eqn:non-skewed-res-bnd}
  \norm{f_{1}}^{\deg f_{2}}_{2}\cdot\norm{f_{2}}^{\deg f_{1}}_{2}\geq N.
\end{equation}
The $2$-norm may greatly over estimate the coefficient size of highly skewed polynomials.  To provide tighter bounds, a generalisation of inequality \eqref{eqn:non-skewed-res-bnd} to the skewed $2$-norm is derived in this section.  To begin, the definition and some properties of the resultant of two polynomials are introduced.

Let $\mathbb{A}$ be a commutative ring (with unity).  Let $f=\sum^{m}_{i=0}a_{i}x^{i}$ and $g=\sum^{n}_{i=0}b_{i}x^{i}$ be non-constant polynomials in $\mathbb{A}[x]$ such that $\deg f=m$ and $\deg g=n$.  The \emph{Sylvester matrix} of $f$ and $g$, denoted $\mathrm{Syl}(f,g)$, is the $(m+n)\times(m+n)$ matrix
\begin{equation*}
\mathrm{Syl}(f,g)=\begin{pmatrix}
	a_{m} & a_{m-1} & \ldots  & \ldots & a_{0}   &         &        &       \\ 
        & a_{m}   & a_{m-1} & \ldots & \ldots  & a_{0}   &        &       \\
	      &         & \ldots  & \ldots & \ldots  & \ldots  & \ldots &       \\
	      &         &         & a_{m}  & a_{m-1} & \ldots & \ldots  & a_{0} \\
	b_{n} & b_{n-1} & \ldots  & \ldots & b_{0}   &         &        &       \\
	      & b_{n}   & b_{n-1} & \ldots & \ldots  & b_{0}   &        &       \\
	      &         & \ldots  & \ldots & \ldots  & \ldots  & \ldots &       \\
	      &         &         & b_{n}  & b_{n-1} & \ldots  & \ldots & b_{0}
\end{pmatrix}
\end{equation*}
where there are $n$ rows containing the $a_{i}$, $m$ rows containing the $b_{i}$, and all empty entries are $0$.  The \emph{resultant} of $f$ and $g$, denoted $\res(f,g)$, is equal to the determinant of the Sylvester matrix $\mathrm{Syl}(f,g)$.  The resultant $\res(f,g)$ has the following well-known properties (see \cite[Chapter~IV, \S8]{lang02}):
\begin{itemize}
	\item If $\mathbb{A}=\C$, $\seq{\alpha}{1}{m}$ are the roots of $f$ and $\seq{\beta}{1}{n}$ the roots of $g$, then
  \begin{equation}\label{res:poisson}
    \res(f,g)=a^{n}_{m}b^{m}_{n}\prod_{i,j}(\alpha_{i}-\beta_{j}).
  \end{equation}
  \item There exist polynomials $u,v\in\mathbb{A}[x]$ such that $uf+vg=\res(f,g)$.
\end{itemize}

For coprime non-constant polynomials $f_{1},f_{2}\in\Z[x]$ with a common root modulo $N$, the second property implies that $N$ divides their resultant, which is nonzero by the first property.  Thus, $N\leq|\res(f_{1},f_{2})|$ and inequality \eqref{eqn:non-skewed-res-bnd} is obtained by using Hadamard's determinant theorem~\cite{hadamard93} to obtain an upper bound on $|\det\mathrm{Syl}(f_{1},f_{2})|$.  Hadamard's theorem is generalised by a result of Fischer~\cite{fischer08}, which states that the determinant of a positive definite Hermitian matrix $H$, partitioned in the form
\begin{equation*}
  H=\begin{pmatrix} H_{1} & X \\ X^{*} & H_{2} \end{pmatrix}
\end{equation*}
such that $H_{1}$ and $H_{2}$ are square matrices, satisfies $\det H\leq\det H_{1}\cdot\det H_{2}$ with equality if and only if $X=0$.  In the following, Fischer's inequality is used to generalise the resultant bound.
\begin{lemma}\label{lem:skewed-resultant-bnd}  Let $f=\sum^{m}_{i=0}a_{i}x^{i}$ and $g=\sum^{n}_{i=0}b_{i}x^{i}$ be non-constant polynomials with real coefficients such that $a_{m}\neq 0$ and $b_{n}\neq 0$.  For $s>0$, define $\theta_{s}=\theta_{s}(f,g)$ to be the angle between the vectors $(a_{i}s^{i})_{0\leq i\leq\max\{m,n\}}$ and $(b_{i}s^{i})_{0\leq i\leq\max\{m,n\}}$, where $a_{i}=0$ if $i>m$, and $b_{i}=0$ if $i>n$.  Then
\begin{equation}\label{eqn:skewed-resultant-bnd}
  \left|\res(f,g)\right|\leq\left|\sin\theta_{s}\right|^{\min\{m,n\}}\cdot\norm{f}^{n}_{2,s}\norm{g}^{m}_{2,s},\quad\text{for all $s>0$}.
\end{equation}
If $m\neq n$, then the inequality is strict for all $s>0$.
\end{lemma}
\begin{proof}  Without loss of generality, since \eqref{eqn:skewed-resultant-bnd} is unaltered by interchanging $f$ and $g$, assume that $m\leq n$.  For all $s>0$, the right-hand side of the inequality in \eqref{eqn:skewed-resultant-bnd} is nonnegative if $m=n$, and positive if $m\neq n$.  Consequently, the lemma holds if $\res(f,g)=0$.  Therefore, assume that $\res(f,g)$ is nonzero.

If $\seq{\alpha}{1}{m}\in\C$ are the roots of $f$ and $\seq{\beta}{1}{n}\in\C$ the roots of $g$, then
\begin{equation}\label{eqn:skewed-resultant-bnd-i}
\begin{split}
  \res(f,g) &= a^{n}_{m}b^{m}_{n}\prod_{i,j}(\alpha_{i}-\beta_{j})%
             = \left(a_{m}s^{\frac{m}{2}}\right)^{n}\left(b_{n}s^{\frac{n}{2}}\right)^{m}\prod_{i,j}\left(\frac{\alpha_{i}}{s}-\frac{\beta_{j}}{s}\right)\\
            &= \res\left(s^{-\frac{m}{2}}f(sx),s^{-\frac{n}{2}}g(sx)\right),\quad\text{for all $s>0$}. 
\end{split}            
\end{equation}
Set $S=\mathrm{Syl}\left(s^{-m/2}f(sx),s^{-n/2}g(sx)\right)$ for some $s>0$, and define $2\times(m+n)$ submatrices $A_{1},\ldots,A_{m}$ of $S$ as follows:  the first and second rows of $A_{i}$ are equal to the $(n-m+i)$th and $(n+i)$th rows of $S$, respectively.  Define the $2m\times 2m$ block matrix $Q_{1}=(A_{i}A^{t}_{j})_{1\leq i,j\leq m}$.  If $m=n$, then there exists a $2m\times 2m$ permutation matrix $P_{1}$ such that $Q_{1}=P_{1}SS^{t}P^{t}_{1}$.  Consequently, \eqref{eqn:skewed-resultant-bnd-i} implies that
\begin{equation}\label{eqn:skewed-resultant-bnd-ii}
  \det Q_{1}=\det(P_{1})^{2}\cdot\det(S)^{2}=\res(f,g)^{2},\quad\text{if $m=n$}.
\end{equation}

The assumption that $\res(f,g)$ is nonzero and \eqref{eqn:skewed-resultant-bnd-i} imply that the rows of $S$ are linearly independent.  Thus, the matrix $Q_{1}$ and its submatrices $Q_{k}=(A_{i}A^{t}_{j})_{k\leq i,j\leq m}$, for $2\leq k\leq m$, are positive definite Hermitian.  Consequently, Fischer's inequality implies that $\det Q_{i}\leq\det A_{i}A^{t}_{i}\cdot\det Q_{i+1}$, for $1\leq i\leq m-1$.  The construction of $A_{1},\ldots,A_{m}$ implies that the matrices $A_{i}A^{t}_{i}$ are equal, with
\begin{equation*}
  A_{i}A^{t}_{i}=\begin{pmatrix}
    \norm{f}^{2}_{2,s} & \norm{f}_{2,s}\norm{g}_{2,s}\cos\theta_{s}\\
    \norm{f}_{2,s}\norm{g}_{2,s}\cos\theta_{s} & \norm{g}^{2}_{2,s}
  \end{pmatrix},\quad\text{for $1\leq i\leq m$}.
\end{equation*}
Hence,
\begin{equation}\label{eqn:skewed-resultant-bnd-iii}
  \det Q_{1}\leq\prod^{m}_{i=1}\det A_{i}A^{t}_{i}=\left(\left(\sin\theta_{s}\right)^{m}\cdot\norm{f}^{m}_{2,s}\norm{g}^{m}_{2,s}\right)^{2}.
\end{equation}
As $s$ was arbitrary, combining \eqref{eqn:skewed-resultant-bnd-ii} and \eqref{eqn:skewed-resultant-bnd-iii} shows that \eqref{eqn:skewed-resultant-bnd} holds if $m=n$.  Therefore, assume that $m\neq n$, and let $A_{0}$ be the $(n-m)\times(m+n)$ submatrix of $S$ consisting of its first $n-m$ rows.  Define the $(m+n)\times(m+n)$ block matrix $Q_{0}=(A_{i}A^{t}_{j})_{0\leq i,j\leq m}$.  Then there exists a $(m+n)\times(m+n)$ permutation matrix $P_{0}$ such that $Q_{0}=P_{0}SS^{t}P^{t}_{0}$.  Thus, \eqref{eqn:skewed-resultant-bnd-i} implies that
\begin{equation*}
  \det Q_{0}=\det(P_{0})^{2}\cdot\det(S)^{2}=\res(f,g)^{2}\neq 0.
\end{equation*}
Therefore, $Q_{0}$ is positive definite Hermitian and, by using Fischer's inequality to bound the determinant of $Q_{0}$, it follows that
\begin{equation}\label{eqn:skewed-resultant-bnd-iv}
  \res(f,g)^{2}\leq\det A_{0}A^{t}_{0}\cdot\det Q_{1}
\end{equation}
with equality if and only if $A_{0}A^{t}_{j}=0$, for $1\leq j\leq m$.  Hadamard's determinant theorem (alternatively, Fischer's inequality) implies that
\begin{equation}\label{eqn:skewed-resultant-bnd-v}
  \det A_{0}A^{t}_{0}\leq\norm{f}^{2(n-m)}_{2,s}
\end{equation}
with equality if and only if the rows of $A_{0}$ are pairwise orthogonal.  Therefore, as $s$ was arbitrary, combining inequalities \eqref{eqn:skewed-resultant-bnd-iii}, \eqref{eqn:skewed-resultant-bnd-iv} and \eqref{eqn:skewed-resultant-bnd-v} yields \eqref{eqn:skewed-resultant-bnd} if $m\neq n$.

Suppose for contradiction that $m\neq n$ and equality holds in \eqref{eqn:skewed-resultant-bnd}.  Then equality must hold in \eqref{eqn:skewed-resultant-bnd-iv} and \eqref{eqn:skewed-resultant-bnd-v}.  Thus, the rows of $A_{0}$ are pairwise orthogonal and $A_{0}A^{t}_{j}=0$, for $1\leq j\leq m$.  It follows that the first row of $S$ is orthogonal to rows $2,\ldots,m+1$ and $n+1$.  Therefore,
\begin{equation*}
  \sum^{k}_{i=0}a_{i}a_{m-k+i}s^{2i-k}=0,\text{ for $0\leq k\leq m-1$;}\quad\text{and}\quad\sum^{m}_{i=0}a_{i}b_{n-m+i}s^{2i+\frac{n-3m}{2}}=0.
\end{equation*}
By using the assumption that $a_{m}$ and $s$ are nonzero to successively eliminate coefficients, it follows that $a_{0},\ldots,a_{m-1}$ and thus $b_{n}$ are all zero.  This is a contradiction since $b_{n}\neq 0$.  Hence, if $m\neq n$, then the inequality \eqref{eqn:skewed-resultant-bnd} is strict for all $s>0$.
\end{proof}

\begin{corollary}\label{cor:skewed-resultant-bnd}  Let $f_{1},f_{2}\in\Z[x]$ be non-constant coprime polynomials with a common root modulo a positive integer $N$.  Then
\begin{equation*}\label{eqn:skewed-res-bnd}
  N\leq\left|\sin\theta_{s}\right|^{\min\{\deg f_{1},\deg f_{2}\}}\cdot\norm{f_{1}}^{\deg f_{2}}_{2,s}\norm{f_{2}}^{\deg f_{1}}_{2,s},\quad\text{for all $s>0$}.
\end{equation*}
If $m\neq n$, then the inequality is strict for all $s>0$.
\end{corollary}

\begin{remark}\label{rmk:skewed-resultant-bnd}  For non-constant polynomials $f,g\in\R[x]$ such that $g=cf$ for some $c\in\R$, the bound in Lemma~\ref{lem:skewed-resultant-bnd} is attained for all  $s>0$:  $\res(f,g)=0$ and $\theta_{s}(f,g)=0$, for all $s>0$.  The bound in Lemma~\ref{lem:skewed-resultant-bnd} is attained for $s>0$ and $d\geq 1$ by polynomials $f_{1}=x^{d}-s^{d}$ and $f_{2}=x^{d}+s^{d}$:  the product formula \eqref{res:poisson} implies that $\res(f_{1},f_{2})=(2s^{d})^{d}$, $\theta_{s}(f_{1},f_{2})=\pi/2$ and $\norm{f_{1}}^{d}_{2,s}\cdot\norm{f_{2}}^{d}_{2,s}=(s^{d/2}\sqrt{2})^{2d}=(2s^{d})^{d}$.  If $d=1$ and $s$ is an integer, then the lower bound in Corollary~\ref{cor:skewed-resultant-bnd} is also attained, since $x-s$ and $x+s$ have a common root modulo $2s$.
\end{remark}

For integers currently within reach of factorisation by the number field sieve, the optimal choice of degree sum $\deg f_{1}+\deg f_{2}$, the main complexity parameter of the algorithm \cite[Section~11]{buhler93}, remains small \cite[Section~3.1]{murphy99}.  For example, the factorisation of a $768$-bit RSA modulus by Kleinjung et al.~\cite{kleinjung10} and the special number field sieve~\cite{lenstra93b} factorisation of $2^{1039}-1$ by Aoki et al.~\cite{aoki07} both used polynomial pairs with a degree sum of $7$.  Corollary~\ref{cor:skewed-resultant-bnd} shows that the restriction to small degree sums implies that a pair of number field sieve polynomials will necessarily have large coefficients.  For large $N$ without special form, how to find polynomials that are close to attaining the lower bound in Corollary~\ref{cor:skewed-resultant-bnd} remains an open problem.
  
\subsection{Lattices in $\R^{n}$}

This section presents necessary background on lattices and lattice reduction.  For further background, see \cite{cohen93,martinet03,lenstra08,nguyen10}.

A \emph{lattice} in $\R^{n}$ is a subgroup $\Lambda$ of $\R^{n}$ with the following property:  there exists $\R$-linearly independent vectors $\seq{\vec{b}}{1}{k}\in\R^{n}$ such that $\Lambda=\sum^{k}_{i=1}\Z\vec{b}_{i}$.  The vectors $\seq{\vec{b}}{1}{k}$ are said to form a \emph{basis} of $\Lambda$, denoted throughout by a $k$-tuple $\mathcal{B}=(\seq{\vec{b}}{1}{k})$, and $k$ is called the \emph{dimension} or \emph{rank} of $\Lambda$.  When written with respect to the canonical orthonormal basis of $\R^{n}$, if $\vec{b}_{i}=(b_{i,1},\ldots,b_{i,n})$ for $1\leq i\leq k$, then the $k\times n$ matrix $B=(b_{i,j})_{1\leq i\leq k,1\leq j\leq n}$ is called a \emph{basis matrix} of $\Lambda$.  The \emph{Gram matrix} of $\mathcal{B}$ is the $k\times k$ symmetric matrix $BB^{t}$.  Let $\mathcal{B}_{1}$ and $\mathcal{B}_{2}$ be bases of $\Lambda$ with respective basis matrices $B_{1}$ and $B_{2}$.  Then there exists a matrix $U\in\mathrm{GL}_{k}(\Z)$ such that $UB_{1}=B_{2}$.  Thus, the Gram matrix of $\mathcal{B}_{2}$ is $Q_{2}=UQ_{1}U^{t}$, where $Q_{1}$ is the Gram matrix of $\mathcal{B}_{1}$.  Therefore, the determinant of the Gram matrix is independent of the choice of basis.  The \emph{determinant} of $\Lambda$ is defined to be $\det\Lambda=\sqrt{\det Q}$, where $Q$ is the Gram matrix of one of its bases.  The zero lattice $\{\vec{0}\}$ is $0$-dimensional and has determinant $1$.

The \emph{sublattices} of a lattice are its subgroups.  If $\Lambda'$ is a sublattice of $\Lambda$ such that its dimension is equal to that of $\Lambda$, then $\Lambda'$ is called a \emph{full-rank sublattice}.  A sublattice $\Lambda'$ of $\Lambda$ is full-rank if and only if $[\Lambda:\Lambda']$ is finite, in which case $\det\Lambda'=[\Lambda:\Lambda']\cdot\det\Lambda$.  Let $\left\langle\vec{x},\vec{y}\right\rangle\mapsto\vec{x}\cdot\vec{y}$ denote the usual inner product on $\R^{n}$.  The \emph{dual lattice} of $\Lambda$ is
\begin{equation*}
  \Lambda^{\times}=\{\vec{x}\in\mathrm{span}(\Lambda)\mid\text{$\left\langle\vec{x},\vec{y}\right\rangle\in\Z$, for all $\vec{y}\in\Lambda$}\}.
\end{equation*}
For any basis $\mathcal{B}$ of $\Lambda$, the dual basis $\mathcal{B}^{\times}$ of $\mathrm{span}(\Lambda)$ is a basis of $\Lambda^{\times}$.  A lattice with $\Lambda^{\times}=\Lambda$ is called \emph{unimodular}.  The lattice $\Z^{n}$ is unimodular.

Let $\norm{.}_{2}$ be the Euclidean norm on $\R^{n}$.  For a $k$-dimensional lattice $\Lambda$ the \emph{$i$th successive minimum} $\lambda_{i}(\Lambda)$ of $\Lambda$, for $1\leq i\leq k$, is defined to be the minimum of $\max_{1\leq j\leq i}\norm{\vec{v}_{j}}_{2}$ over all linearly independent lattice vectors $\seq{\vec{v}}{1}{i}\in\Lambda$.  Minkowski's second theorem (see \cite[p.~35]{nguyen10}) states that the geometric mean of the first $t$ successive minima is at most $\sqrt{\gamma_{k}}\det(\Lambda)^{\frac{1}{k}}$, for $1\leq t\leq k$, where $\gamma_{k}$ is Hermite's constant (see \cite[p.~20]{nguyen10}). 

An LLL-reduced basis $(\seq{\vec{b}}{1}{k})$ of a lattice $\Lambda$ has the property that $\norm{\vec{b}_{i}}_{2}$ approximates the $i$th successive minimum $\lambda_{i}(\Lambda)$ for $1\leq i\leq k$:
\begin{theorem}\label{thm:LLL-properties}   Let $(\seq{\vec{b}}{1}{k})$ be an LLL-reduced basis of a $k$-dimensional lattice $\Lambda\subset\R^{n}$.  Then
\begin{enumerate}
  \item\label{lll:bi-minima-bnd} $\norm{\vec{b}_{i}}_{2}\leq 2^{(k-1)/2}\lambda_{i}(\Lambda)$ for $1\leq i\leq k$; and
  \item\label{lll:bi-bnd} if $\Lambda\subseteq\Z^{n}$, then $\norm{\vec{b}_{i}}_{2}\leq 2^{\frac{k(k-1)}{4(k-i+1)}}\det\Lambda^{\frac{1}{k-i+1}}$ for $1\leq i\leq k$.
\end{enumerate}
\end{theorem}
Property~\eqref{lll:bi-minima-bnd} of Theorem~\ref{thm:LLL-properties} is due to Lenstra, Lenstra and Lov\'as~\cite{lll}, and property~\eqref{lll:bi-bnd} is due to May~\cite[Theorem~4]{may03}.  The $\mathrm{L}^{2}$ algorithm~\cite{nguyen09,nguyen05} transforms an arbitrary basis $(\seq{\vec{b}}{1}{k})$ of a $k$-dimensional lattice $\Lambda\subseteq\Z^{n}$ into an LLL-reduced basis in $O(k^{4}n(k+\log\beta)\log\beta)$ bit operations, where $\beta=\max_{1\leq i\leq k}\norm{\vec{b}_{i}}_{2}$.

\section{Nonlinear polynomial selection}\label{sec:gp-review}

Nonlinear polynomial generation algorithms are based on the observation that bounded degree integer polynomials with a prescribed root modulo $N$ are characterised by an orthogonality condition on their coefficient vectors:  an integer polynomial $f=\sum^{d}_{i=0}a_{i}x^{i}$ of degree \emph{at most} $d$ has $m$ as a root modulo $N$ if and only if the \emph{coefficient vector} $(\seq{a}{0}{d})$ is orthogonal to $(1,m,\ldots,m^{d})$ modulo $N$.  The set of all such coefficient vectors, denoted $L_{m,d}$, forms a lattice in $\Z^{d+1}$ \cite[Section~12.2]{buhler93}.  Nonlinear algorithms employ LLL-reduction to search for short vectors in sublattices of $L_{m,d}$.  Theorem~\ref{thm:LLL-properties} suggests that this approach yields polynomials with small coefficients whenever the sublattices have small determinants.

Using an approach introduced by Montgomery (see \cite[Section~5]{elkenbracht96} and \cite[Section~2.3.1]{murphy99}), and since applied by several authors \cite{montgomery93,montgomery06,williams10,prest10,koo11}, nonlinear algorithms construct sublattices of $L_{m,d}$ with small determinants from ``small" geometric progressions modulo $N$.  A \emph{geometric progression} (GP) of \emph{length} $\ell$ and \emph{ratio} $r$ modulo $N$, denoted throughout by a vector $[\seq{c}{0}{\ell-1}]$, is an integer sequence with the property that $c_{i}\equiv c_{0}r^{i}\pmod{N}$, for $0\leq i<\ell$.  Central to the construction of lattices for nonlinear algorithms is the observation that
\begin{equation*}
  L_{m,d}=\left\{(\seq{a}{0}{d})\in\Z^{d+1}\mid\sum^{d}_{i=0}a_{i}c_{i}\equiv 0\bmod{N}\right\},
\end{equation*}
for any length $d+1$ GP $[\seq{c}{0}{d}]$ with ratio $m$ modulo $N$, nonzero terms and $\gcd(c_{0},N)=1$.  Given such a GP, nonlinear algorithms consider a sublattice of $L_{m,d}$ contained in the $\Q$-vector space orthogonal to $[\seq{c}{0}{d}]$.  The role of $N$ in the definition of the sublattice is therefore made implicit, resulting in its determinant being dependent on the terms of the GP and not on $N$ itself.  Consequently, a GP with terms that are small when compared to $N$ is expected to give rise to a sublattice of $L_{m,d}$ with small determinant.  More generally, lattices contained in the $\Q$-vector space orthogonal to multiple linearly independent geometric progressions are considered.

There are two main problems that immediately arise from this approach: firstly, establishing a relationship between the size of terms in the geometric progressions and the determinant of the resulting lattices; and secondly, constructing geometric progressions with small terms.  To help address the first problem, some properties of orthogonal lattices are introduced in the next section.  Section~\ref{sec:gp-detail} takes a closer look at nonlinear polynomial generation and, based on the results of Section~\ref{sec:ortho-lattice}, provides criteria for the selection of geometric progressions.  In Section~\ref{sec:gp-exist}, existing solutions to the second problem are reviewed. 

In this paper, big-$O$ estimates may have implied constants depending on the degree parameter $d$, and $N$ denotes a positive integer condemned to factorisation.

\subsection{The orthogonal lattice}\label{sec:ortho-lattice}

For a lattice $\Lambda\subseteq\Z^{n}$, denote by $E_{\Lambda}$ the unique $\Q$-vector subspace of $\Q^{n}$ that is generated by any of its bases.  The dimension of $E_{\Lambda}$ over $\Q$ is equal to the dimension of $\Lambda$.  Let $E^{\bot}_{\Lambda}$ denote the orthogonal complement of $E_{\Lambda}$ with respect to $\left\langle\,\,\,,\,\,\right\rangle$.  The \emph{orthogonal lattice} of $\Lambda$ is defined to be $\Lambda^{\bot}=\Z^{n}\cap E^{\bot}_{\Lambda}$.  A result of Martinet~\cite[Proposition~1.3.4]{martinet03} implies that $\dim\Lambda^{\bot}=\dim E^{\bot}_{\Lambda}$ if and only if $\dim(\Z^{n})^{\times}\cap E^{\bot\bot}_{\Lambda}=\dim E_{\Lambda}$.  The latter holds since $(\Z^{n})^{\times}\cap E^{\bot\bot}_{\Lambda}=\Z^{n}\cap E_{\Lambda}$ is a lattice (see \cite[Proposition~1.1.3]{martinet03}) which contains $\Lambda$.  Hence, $\dim\Lambda+\dim\Lambda^{\bot}=n$.

For a lattice $\Lambda\subseteq\Z^{n}$, denote by $\overline{\Lambda}$ the lattice $\Z^{n}\cap E_{\Lambda}$.  Nguyen and Stern~\cite[p.~200]{nguyen97} show that $\det\Lambda=\left[\,\overline{\Lambda}:\Lambda\,\right]\cdot\det\Lambda^{\bot}$.  A lattice $\Lambda\subseteq\Z^{n}$ is called \emph{primitive} if $\overline{\Lambda}=\Lambda$.  A $k$-dimensional lattice $\Lambda\subseteq\Z^{n}$ with basis matrix $B$ is primitive if and only if the greatest common divisor of all $k\times k$ minors of $B$ is $1$ (see \cite[Corollary 4.1c]{schrijver}).  The following lemma determines the index $\left[\,\overline{\Lambda}:\Lambda\,\right]$ in general:
\begin{lemma}\label{lem:index-incomplete-lattice}  Let $\Lambda\subseteq\Z^{n}$ be a $k$-dimensional lattice with basis matrix $B$.  If $\Omega$ is the greatest common divisor of all $k\times k$ minors of $B$, then $\left[\,\overline{\Lambda}:\Lambda\,\right]=\Omega$.
\end{lemma}
\begin{proof}  Let $\overline{B}$ be a basis matrix of $\overline{\Lambda}$.  The lattice $\Lambda$ is a full-rank sublattice of $\overline{\Lambda}$, thus there exists a $k\times k$ integer matrix $U$ such that $|\det U|=\left[\,\overline{\Lambda}:\Lambda\,\right]$ and $B=U\cdot\overline{B}$.  Hence, the lemma will follow by showing that $\Omega=|\det U|$.  

For indices $1\leq i_{1}<\ldots<i_{k}\leq n$, let $B_{i_{1},\ldots,i_{k}}$ denote the $k\times k$ submatrix $B$ formed by columns $i_{1},\ldots,i_{k}$.  Similarly, let $\overline{B}_{i_{1},\ldots,i_{k}}$ denote the submatrix of $\overline{B}$ formed by distinct columns $i_{1},\ldots,i_{k}$.  Then $B_{i_{1},\ldots,i_{k}}=U\cdot\overline{B}_{i_{1},\ldots,i_{k}}$, for all $1\leq i_{1}<\ldots<i_{k}\leq n$.  Therefore, $\Omega=|\det U|\cdot\overline{\Omega}$, where $\overline{\Omega}$ is the greatest common divisor of all $k\times k$ minors of $\overline{B}$.  However, $\overline{\Omega}=1$ as the lattice $\overline{\Lambda}$ is primitive.
\end{proof}

\subsubsection{The determinant under transformation}

For a $k$-dimensional lattice $\Lambda\subseteq\Z^{n}$ and $S\in\mathrm{GL}_{n}(\R)$, define $\Lambda_{S}=\{\vec{x}\cdot S\mid\vec{x}\in\Lambda\}$.  Given a basis $(\seq{\vec{b}}{1}{k})$ of $\Lambda$, define $(\seq{\vec{b}}{1}{k})_{S}=(\vec{b}_{1}S,\ldots,\vec{b}_{k}S)$.  Then $\Lambda_{S}$ is a $k$-dimensional lattice in $\R^{n}$ with basis $(\seq{\vec{b}}{1}{k})_{S}$.
\begin{lemma}\label{lem:det-scaled-ortho} Let $\Lambda$ be a lattice in $\Z^{n}$ and $S\in\mathrm{GL}_{n}(\R)$.  Then
\begin{equation*}
  \det\Lambda^{\bot}_{S}=\left|\det S\right|\cdot\det\overline{\Lambda}_{S^{-t}},
\end{equation*}
where $S^{-t}=(S^{-1})^{t}$ denotes the inverse transpose of $S$.
\end{lemma}
\begin{proof}  Fix a basis $(\seq{\vec{b}}{1}{k})$ of $\overline{\Lambda}$.  The lattice $\overline{\Lambda}$ is primitive, thus $(\seq{\vec{b}}{1}{k})$ can be extended to a basis $(\seq{\vec{b}}{1}{n})$ of $\Z^{n}$ \cite[Lemma 2, Chapter 1]{cassels71}.  Since $\Z^{n}$ is unimodular, the dual basis $(\seq{\vec{b}^{\times}}{1}{n})$ of $\R^{n}$ forms a basis of $\Z^{n}$.  The dual basis is characterised by the equalities $\langle\vec{b}^{\times}_{i},\vec{b}_{j}\rangle=\delta_{i,j}$, for $1\leq i,j\leq n$, where $\delta_{i,j}$ is the Kronecker delta.  Therefore, $(\seq{\vec{b}^{\times}}{k+1}{n})$ forms a basis of the orthogonal lattice $\Lambda^{\bot}$.  Hence, $(\seq{\vec{b}}{1}{n})_{S^{-t}}$ forms a basis of $\Z^{n}_{S^{-t}}$, $(\seq{\vec{b}}{1}{k})_{S^{-t}}$ forms a basis of $\overline{\Lambda}_{S^{-t}}$ and $(\seq{\vec{b}^{\times}}{k+1}{n})_{S}$ forms a basis of $\Lambda^{\bot}_{S}$.

For all $1\leq i,j\leq n$,
\begin{equation*}
  \langle\vec{b}^{\times}_{i}S,\vec{b}_{j}S^{-t}\rangle=\vec{b}^{\times}_{i}SS^{-1}{\vec{b}_{j}}^{t}=\langle\vec{b}^{\times}_{i},\vec{b}_{j}\rangle=\delta_{i,j}.
\end{equation*}
Thus $(\seq{\vec{b}^{\times}}{1}{n})_{S}$ is a dual basis of $(\seq{\vec{b}}{1}{n})_{S^{-t}}$.  Therefore, by applying a result of Martinet~\cite[Corollary~1.3.5]{martinet03}, with $E=\R^{n}$ and $F$ equal to the subspace of $\R^{n}$ generated by $(\seq{\vec{b}}{1}{k})_{S^{-t}}$, it follows that
\begin{equation*}
  \left|\det S\right|^{-1} = \det\Z^{n}_{S^{-t}} =\det\left(\overline{\Lambda}_{S^{-t}}\right)\cdot\det\big(\Lambda^{\bot}_{S}\big)^{-1}.\qedhere
\end{equation*}
\end{proof}

Given a basis of a lattice $\Lambda\subseteq\Z^{n}$ and a diagonal matrix $S\in\mathrm{GL}_{n}(\R)$, the following theorem provides a method for computing the determinant of $\Lambda^{\bot}_{S}$:
\begin{theorem}\label{thm:det-ortho-lattice}  Let $\Lambda\subseteq\Z^{n}$ be a $k$-dimensional lattice with basis matrix $B$.  For all indices $1\leq i_{1}<\ldots<i_{k}\leq n$, denote by $B_{i_{1},\ldots,i_{k}}$ the $k\times k$ submatrix of $B$ formed by columns $i_{1},\ldots,i_{k}$.  Suppose $S=\mathrm{diag}(\seq{S}{1}{n})$ for nonzero real numbers $\seq{S}{1}{n}$.  Then
\begin{equation*}
  \det\Lambda^{\bot}_{S}=\frac{\left|S_{1}\cdots S_{n}\right|}{\Omega}\cdot\sqrt{\sum_{1\leq i_{1}<\ldots<i_{k}\leq n}\left(\frac{\det B_{i_{1},\ldots,i_{k}}}{S_{i_{1}}\cdots S_{i_{k}}}\right)^{2}},
\end{equation*}
where $\Omega$ is the greatest common divisor of all $k\times k$ minors of $B$.  
\end{theorem}
\begin{proof}  The index of $\Lambda$ in $\overline{\Lambda}$ is invariant under scaling by the matrix $S^{-1}$, i.e., $\left[\,\overline{\Lambda}_{S^{-1}}:\Lambda_{S^{-1}}\right]=\left[\,\overline{\Lambda}:\Lambda\,\right]$.  Therefore, Lemma~\ref{lem:index-incomplete-lattice} and Lemma~\ref{lem:det-scaled-ortho} imply that
\begin{equation*}
  \det\Lambda^{\bot}_{S}=\left|\det S\right|\cdot\det\overline{\Lambda}_{S^{-1}}=\left|S_{1}\cdots S_{n}\right|\cdot\Omega^{-1}\cdot\det\Lambda_{S^{-1}}.
\end{equation*}
The matrix $P=BS^{-1}$ forms a basis matrix of $\Lambda_{S^{-1}}$.  Therefore, by using the Cauchy--Binet formula (see \cite[p.~86]{aitken56}) to compute $\det PP^{t}$, it follows that
\begin{equation*}
  \det\Lambda_{S^{-1}}=\sqrt{\sum_{1\leq i_{1}<\ldots<i_{k}\leq n}\det\left(P_{i_{1},\ldots,i_{k}}\right)^{2}},
\end{equation*}
where $P_{i_{1},\ldots,i_{k}}=B_{i_{1},\ldots,i_{k}}\cdot\mathrm{diag}(S_{i_{1}},\ldots,S_{i_{k}})^{-1}$ is the $k\times k$ submatrix of $P$ formed by columns $i_{1},\ldots,i_{k}$, for all $1\leq i_{1}<\ldots<i_{k}\leq n$.
\end{proof}

\subsubsection{Computing a basis of the orthogonal lattice}\label{sec:ortho-lattice-basis}
Let $\Lambda$ be a $k$-dimensional lattice in $\Z^{n}$ and $B=(b_{i,j})_{1\leq i\leq k,1\leq j\leq n}$ be one of its basis matrices.  A basis of the orthogonal lattice $\Lambda^{\bot}$ is found by using Algorithm~2.4.10 or Algorithm~2.7.2 of Cohen~\cite{cohen93} to compute a basis of the integer kernel of $B$.  The former algorithm is based on Hermite normal form computation (see \cite[Section~2.4.2]{cohen93}) and the latter algorithm on the MLLL algorithm of Pohst~\cite{pohst87}.  In practice, the MLLL based algorithm is preferable, since it is more likely to avoid large integer arithmetic~\cite[Section~2.4.3]{cohen93}.  Similarly, the LLL HNF algorithm of Havas, Majewski and Matthews~\cite[Section~6]{havas98} can be used.  If $\beta=\max_{j}\norm{(b_{1,j},\ldots,b_{k,j})}^{2}_{2}$, then the algorithm performs $O((n+k)^{4}\log(n\beta))$ operation on integers of size $O(n\log(n\beta))$~\cite{vanderkallen00}.  The algorithm of Nguyen and Stern~\cite[Algorithm~5]{nguyen97} computes an LLL-reduced basis of $\Lambda^{\bot}$.

The algorithm of Nguyen and Stern is readily modified to produce an LLL-reduced basis of $\Lambda^{\bot}_{S}$ for any nonsingular $n\times n$ integer matrix $S$.  Consider the lattice $\Delta$ with basis given by the $n\times(n+k)$ block matrix $D=\begin{pmatrix} S & X B^{t}\end{pmatrix}$, where $X$ is a positive integer.  A vector of norm less than $X$ in $\Delta$ is of the form $\vec{y}D=(\vec{y}S,\vec{0})$ for some $\vec{y}\in\Lambda^{\bot}$.  Thus, an LLL-reduced basis of $\Lambda^{\bot}_{S}$ is obtained by first computing an LLL-reduced basis $(\seq{\vec{x}}{1}{n})$ of $\Delta$ for a sufficiently large value of $X$, and then projecting the basis vectors $\seq{\vec{x}}{1}{n-k}$ onto their first $n$ consecutive entries.  The existence of an LLL-reduced basis for $\Lambda^{\bot}_{S}$ and property~\ref{lll:bi-bnd} of Theorem~\ref{thm:LLL-properties} imply that $\Delta$ contains linearly independent vectors $\seq{\vec{y}}{1}{n-k}$ such that
\begin{equation*}
  \max_{1\leq i\leq n-k}\norm{\vec{y}_{i}}_{2}\leq 2^{\frac{(n-k)(n-k-1)}{4}}\det\Lambda^{\bot}_{S}.
\end{equation*}
Thus, property~\ref{lll:bi-minima-bnd} of Theorem~\ref{thm:LLL-properties} implies that it is sufficient to choose $X$ such that 
\begin{equation*}
  X>2^{\frac{n-1}{2}+\frac{(n-k)(n-k-1)}{4}}\det\Lambda^{\bot}_{S}.
\end{equation*}
Nguyen and Stern observe that theoretical bounds on LLL-reduced bases such as those in Theorem~\ref{thm:LLL-properties} are ``quite pessimistic'' in practical circumstances.  Consequently, a smaller value of $X$ may often be used in practice.

\subsection{Nonlinear polynomial generation}\label{sec:gp-detail}

Nonlinear algorithms first construct $k$, $1\leq k<d$, linearly independent geometric progressions
\begin{equation*}
  \vec{c}_{1}=[c_{1,0},\ldots,c_{1,d}], \vec{c}_{2}=[c_{2,0},\ldots,c_{2,d}],\ldots,\vec{c}_{k}=[c_{k,0},\ldots,c_{k,d}]
\end{equation*}
with ratio $m$ modulo $N$, then use lattice reduction to search for polynomials with coefficient vectors that are short vectors in the lattice $L^{\bot}_{S}$, where $L$ is the $k$-dimensional lattice with basis $(\seq{\vec{c}}{1}{k})$ and $S=\mathrm{diag}(1,s,\ldots,s^{d})$ for some positive skew $s$.  In particular, a pair of polynomials is usually obtained from the first two basis vectors of an LLL-reduced basis of $L^{\bot}_{S}$. The lattice $L^{\bot}$ is required to be a sublattice of $L_{m,d}$, so that each vector in the lattice corresponds to a polynomial that admits $m$ as a root modulo $N$.  This requirement fails in general, but is satisfied whenever at least one $\vec{c}_{i}$ has nonzero terms and $\gcd(c_{i,0},N)=1$.  Furthermore, no $\vec{c}_{i}$ should be a rational GP, otherwise each polynomial that corresponds to a vector in $L^{\bot}$ is reducible.

A vector $\vec{a}=(a_{0},a_{1}s\ldots,a_{d}s^{d})$ in the lattice $L^{\bot}_{S}$ corresponds to an integer polynomial $f=\sum^{d}_{i=0}a_{i}x^{i}$ such that $\norm{f}_{2,s}=s^{-(\deg f)/2}\norm{\vec{a}}_{2}$.  Therefore, if the first two basis vectors of an LLL-reduced basis of $L^{\bot}_{S}$ correspond to coprime degree $d$ polynomials $f_{1},f_{2}\in\Z[x]$, then Corollary~\ref{cor:skewed-resultant-bnd} and Theorem~\ref {thm:LLL-properties} imply that
\begin{equation}\label{eqn:nfs-pol-gp-gen-bnd}
  \frac{N^{\frac{1}{d}}}{|\sin\theta_{s}(f_{1},f_{2})|}\leq\norm{f_{1}}_{2,s}\cdot\norm{f_{2}}_{2,s}\leq\frac{2^{d-k}\,\gamma_{d-k+1}}{s^{d}}\cdot\det(L^{\bot}_{S})^{\frac{2}{d-k+1}},
\end{equation}
where $\theta_{s}$ is defined as in Lemma~\ref{lem:skewed-resultant-bnd} and $\gamma_{n}\leq 1+n/4$, for all $n\geq 1$ \cite[p.~17]{milnor73}.  Furthermore, it follows from the definition of an LLL-reduced basis~\cite[Section~1]{lll} that $|\sin\theta_{s}(f_{1},f_{2})|\geq\sqrt{2/3}$.  Consequently, for the purpose of generating two degree $d$ polynomials, the determinant of $L^{\bot}_{S}$ is considered to have optimal size whenever $s^{-d(d-k+1)/2}\det L^{\bot}_{S}$ is $O(N^{(d-k+1)/2d})$.

Nonlinear algorithms attempt to construct geometric progressions such that the determinant of $L^{\bot}_{S}$ is small by constructing ``small" progressions.  For $\vec{x}\in\R^{d+1}$ and any real number $s>0$, define $\norm{\vec{x}}_{2,s}=s^{-d/2}\norm{\vec{x}\Sigma}_{2}$, where $\Sigma=\mathrm{diag}(1,s,\ldots,s^{d})$.  Then the following theorem shows that the determinant of $L^{\bot}_{S}$ is small for geometric progressions $\seq{\vec{c}}{1}{k}$ such that $\norm{\vec{c}_{1}}_{2,s^{-1}},\ldots,\norm{\vec{c}_{k}}_{2,s^{-1}}$ are small:
\begin{theorem}\label{thm:det-mult-gp}  Let $d$ and $k$ be integers such that $1\leq k\leq d$, and suppose that
\begin{equation*}
  \vec{c}_{1}=[c_{1,0},\ldots,c_{1,d}], \vec{c}_{2}=[c_{2,0},\ldots,c_{2,d}],\ldots,\vec{c}_{k}=[c_{k,0},\ldots,c_{k,d}],
\end{equation*}
are linearly independent geometric progressions, each with ratio $m$ modulo $N$, such that $\gcd(c_{1,0},N)=1$.  Let $L\subseteq\Z^{d+1}$ be the lattice with basis $(\seq{\vec{c}}{1}{k})$ and $S=\mathrm{diag}(1,s,\ldots,s^{d})$ for some positive real number $s$.  Then the lattice $L^{\bot}_{S}$ is $(d-k+1)$-dimensional and
\begin{equation}\label{eqn:det-mult-gp}
  \det L^{\bot}_{S}\leq s^{\frac{d(d-k+1)}{2}}\cdot N^{1-k}\cdot\prod^{k}_{i=1}\norm{\vec{c}_{i}}_{2,s^{-1}}.
\end{equation}
\end{theorem}
\begin{proof}  If $k\geq 2$, then $\vec{c}_{i}-(c^{\phantom{1}}_{i,0}/c_{1,0})\vec{c}_{1}\equiv\vec{0}\pmod{N}$, for $2\leq i\leq k$.  Thus, $N^{k-1}$ divides each $k\times k$ minor of the basis matrix $C=(c_{i,j})_{1\leq i\leq k,0\leq j\leq d}$ of $L$.  Hence, Lemma~\ref{lem:index-incomplete-lattice} and Lemma~\ref{lem:det-scaled-ortho} imply that $L^{\bot}_{S}$ is a $(d-k+1)$-dimensional lattice and
\begin{equation}\label{eqn:det-mult-gp-i}
  \det L^{\bot}_{S}\leq\left|\det S\right|\cdot N^{1-k}\cdot\det L_{S^{-1}}=s^{\binom{d+1}{2}}\cdot N^{1-k}\cdot\det L_{S^{-1}}.
\end{equation}
The matrix $CS^{-1}$ is a basis matrix of $L_{S^{-1}}$. Thus, by applying Hadamard's determinant theorem, it follows that
\begin{equation*}
  \det L_{S^{-1}}\leq\prod^{k}_{i=1}\norm{\vec{c}_{i}S^{-1}}_{2}=\prod^{k}_{i=1}s^{-\frac{d}{2}}\cdot\norm{\vec{c}_{i}}_{2,s^{-1}}=s^{-\frac{dk}{2}}\cdot\prod^{k}_{i=1}\norm{\vec{c}_{i}}_{2,s^{-1}}.
\end{equation*}
Combining this inequality with \eqref{eqn:det-mult-gp-i} yields \eqref{eqn:det-mult-gp}.
\end{proof}

To construct multiple geometric progressions with a common ratio modulo $N$, Montgomery~\cite{montgomery93,montgomery06} suggests constructing an initial GP of length $2d-1$ from which $d-1$ geometric progressions of length $d+1$ are obtained by taking subsequences of successive terms.  More generally, Koo, Jo and Kwon~\cite[Section~3]{koo11} suggest constructing an initial GP $\vec{c}=[\seq{c}{0}{\ell-1}]$ of length $\ell$, where $d<\ell<2d$, from which $\ell-d$ geometric progressions of length $d+1$ are obtained:
\begin{equation*}
  \vec{c}_{1}=[c_{0},\ldots,c_{d}],\vec{c}_{2}=[c_{1},\ldots,c_{d+1}],\ldots,\vec{c}_{\ell-d}=[c_{\ell-d-1},\ldots,c_{\ell-1}].
\end{equation*}
If the vectors $\vec{c}_{1},\ldots,\vec{c}_{\ell-d}$ do not form a basis for an $(\ell-d)$-dimensional sublattice of $L_{m,d}$, then $\vec{c}$ is rejected.  For $s>0$, the product of the norms  $\norm{\vec{c}_{i}}_{2,s^{-1}}$ is bounded in terms of the norm of the initial GP:
\begin{equation*}
  \prod^{\ell-d}_{i=1}\norm{\vec{c}_{i}}_{2,s^{-1}}=\prod^{\ell-d}_{i=1}s^{\frac{\ell-d-1}{2}-(i-1)}\cdot\norm{\vec{c}_{i}}_{2,s^{-1}}\leq\norm{\vec{c}}^{\ell-d}_{2,s^{-1}}.
\end{equation*}
Therefore, to generate two degree $d$ polynomials of optimal size, \eqref{eqn:nfs-pol-gp-gen-bnd} and Theorem~\ref{thm:det-mult-gp} suggest constructing an initial GP $\vec{c}$ such that
\begin{equation}\label{eqn:optimal-gp-norm}
  \norm{\vec{c}}_{2,s^{-1}}=O\left(N^{\frac{(2d-1)(\ell-d)-(d-1)}{2d(\ell-d)}}\right),
\end{equation}
for some $s>0$.  For fixed $d$, the weakest size requirements on $\vec{c}$ occur when $\ell=2d-1$, corresponding to Montgomery's algorithm.  For this case, the orthogonal lattice is $2$-dimensional.  Thus, rather than computing an LLL-reduced basis of $L^{\bot}_{S}$, a basis $(\vec{b}_{1},\vec{b}_{2})$ such that $\norm{\vec{b}_{1}}_{2}=\lambda_{1}(L^{\bot}_{S})$ and $\norm{\vec{b}_{2}}_{2}=\lambda_{2}(L^{\bot}_{S})$ is computed in polynomial time with Lagrange's algorithm~\cite{lagrange73} (see also \cite[p.~41]{nguyen10}).  The polynomials $f_{1}$ and $f_{2}$ corresponding to the basis vectors then satisfy $|\sin\theta_{s}(f_{1},f_{2})|\geq\sqrt{3}/2$ (see \cite[p.~41]{nguyen10}).  For large $N$, the problem of how to efficiently constructing geometric progressions that satisfy \eqref{eqn:optimal-gp-norm} remains open for all parameters $(d,\ell)\neq(2,3)$.

Koo, Jo and Kwon~\cite[Section~3]{koo11} observe that a length $\ell$ GP can be used to generate at least one degree $d$ polynomial for all $\ell/2\leq d<\ell$.  Consequently, distinct degree polynomial pairs with a common root modulo $N$ are obtained by varying the parameter $d$.  This observation provides a method of producing polynomial pairs with odd degree sum, allowing nonlinear algorithms to be applied to $N$ of any size.
		    
\subsection{Existing algorithms}\label{sec:gp-exist}

The development of nonlinear generation algorithms has been, for the most part, evolutionary, with successive algorithms building on their predecessors.  The progression of improvements and generalisations in existing nonlinear algorithms is reviewed in this section.

\subsubsection{Montgomery's two quadratics algorithm}

Montgomery's two quadratics algorithm (see \cite[Section~5]{elkenbracht96} and \cite[Section~2.3.1]{murphy99}) uses a length $d+1=3$ GP construction.  The construction begins with the selection of an integer $p\geq 2$, which is usually chosen to be prime, such that $\gcd(p,N)=1$ and $N$ is quadratic residue modulo $p$.  Then one of the values $m\in\Z$ such that $m^{2}\equiv N\pmod{p}$ and $|m-N^{1/2}|\leq p/2$ is chosen.  Finally, the GP is taken to be $[c_{0},c_{1},c_{2}]=[p,m,(m^{2}-N)/p]$, which has ratio $m/p$ modulo $N$.  For a positive integer skew $s$ and an integer $t\equiv c_{2}/c_{1}\pmod{c_{0}}$, LLL-reduction is performed on the row vectors of the matrix
\begin{equation*}\label{eqn:montgomery-M}
  \left(\begin{array}{ccccc}
   c_{1}                      & -c_{0}s & 0    \\
   \frac{c_{1}t-c_{2}}{c_{0}} & -ts     & s^{2}  
\end{array}\right),
\end{equation*}
which is a basis matrix of the lattice $\left([c_{0},c_{1},c_{2}]\Z\right)^{\bot}_{S}$, where $S=\mathrm{diag}(1,s,s^{2})$.

For a given skew $s>0$, choosing $p=O(s^{-1}\sqrt{N})$ guarantees that \eqref{eqn:optimal-gp-norm} holds.  As a result, Montgomery's algorithm is capable of producing polynomials with optimal coefficient size.  However, the restriction to quadratic polynomials means that the algorithm is not suitable for $N$ containing more than 110--120 digits \cite[Section~2.3.1]{murphy99}.  Examples of polynomials generated by Montgomery's two quadratics algorithm are provided by Elkenbracht-Huizing~\cite[Section~10]{elkenbracht96}.

\subsubsection{The Williams and Prest--Zimmermann algorithms}\label{sec:williams-prest-zimmermann}

Williams~\cite[Chapter~4]{williams10} introduces an additional length $3$ GP construction for producing pairs of quadratic polynomials.  Roughly speaking, Williams' construction is the specialisation of Montgomery's construction obtained by setting $p=1$.  Williams additionally provides a length $4$ GP construction for producing pairs of cubic polynomials.  In both of Williams' algorithms, the skew parameter is restricted to $s=1$.  Prest and Zimmermann~\cite{prest10} extended Williams' algorithms to skews $s\neq 1$, leading to a reduction in coefficient norms for the cubic algorithm.  In addition, they generalised their algorithm to arbitrary degrees.

In the algorithms of Williams and Prest--Zimmermann, geometric progressions of length $d+1$ are constructed by first selecting an integer $m$ such that $|m^{d}-N|=O(N^{1-1/d})$.  Then the GP is taken to be
\begin{equation*}
  [c_{0},\ldots,c_{d}]=[1,m,\ldots,m^{d-1},m^{d}-N],
\end{equation*}
which has ratio $m$ modulo $N$.  For a positive integer skew $s$, taken equal to one in Williams' algorithm, LLL-reduction is performed on the row vectors of the matrix
\begin{equation*}\label{eqn:prest-zimmermann-M}
\left(\begin{array}{ccccc}
  -c_{1}  & s      & 0      & \ldots & 0      \\
  -c_{2}  & 0      & s^{2}  & \ldots & 0      \\
  \vdots  & \vdots & \vdots & \ddots & \vdots \\
  -c_{d}  & 0      & 0      & \ldots & s^{d}
\end{array}\right),
\end{equation*}
which is a basis matrix of the lattice $\left([c_{0},\ldots,c_{d}]\Z\right)^{\bot}_{S}$, where $S=\mathrm{diag}(1,s,\ldots,s^{d})$.

Examples of polynomials generated by the Williams and Prest--Zimmermann algorithms are found in \cite[Chapter~5]{williams10} and \cite{prest10}.

\subsubsection{The Koo--Jo--Kwon algorithms}

Koo, Jo and Kwon~\cite[Section~4.1]{koo11} generalise Montgomery's GP construction to arbitrary degrees.  They construct geometric progressions of length $d+1$ by first selecting positive integers $p=O((kN)^{1/d})$ and $k=O(1)$ such that $x^{d}\equiv kN\pmod{p}$ has a nonzero solution.  An integer $m$ satisfying $m^{d}\equiv kN\pmod{p}$ and $|m-\sqrt[d]{kN}|\leq p/2$ is chosen.  Then the GP is taken to be
\begin{equation*}
  [c_{0},\ldots,c_{d}]=\left[p^{d-1},p^{d-2}m,\ldots,m^{d-1},\frac{m^{d}-kN}{p}\right],
\end{equation*}
which has ratio $m/p$ modulo $N$.  Koo, Jo and Kwon use the modified Nguyen--Stern algorithm described in Section~\ref{sec:ortho-lattice-basis} to compute an LLL-reduced basis of the lattice $\left([c_{0},\ldots,c_{d}]\Z\right)^{\bot}_{S}$, with $S=\mathrm{diag}(1,s,\ldots,s^{d})$ for a positive integer skew $s$.

The Koo--Jo--Kwon GP construction reduces to Montgomery's construction for parameters $d=2$, $k=1$, and the constructions of Williams and Prest--Zimmerman for $p=k=1$.  The Koo--Jo--Kwon and Prest--Zimmermann algorithms produce polynomials which satisfy the same theoretical bounds on coefficient norms (see Section~\ref{sec:gp-new-param}).  However, the additional parameters $p$ and $k$ in the Koo--Jo--Kwon construction permit a greater number of geometric progressions to be constructed for any given $N$, which may be leveraged in practical circumstances to find polynomials with smaller coefficients.

By extending their length $d+1$ GP construction, Koo, Jo and Kwon~\cite[Section~4.2]{koo11} obtain a length $d+2$ construction.  The construction begins with the selection of positive integers $p=\Theta((kN)^{1/d})$ and $k=O(1)$ such that $x^{d}\equiv kN\pmod{p^{2}}$ has a nonzero solution $m=\Theta(p)$.  Then the GP is taken to be
\begin{equation*}
  [c_{0},\ldots,c_{d+1}]=\left[p^{d-1},p^{d-2}m,\ldots,m^{d-1},\frac{m^{d}-kN}{p},\frac{m(m^{d}-kN)}{p^{2}}\right],
\end{equation*}
which has ratio $m/p$ modulo $N$.  Koo, Jo and Kwon do not analyse their algorithm for skews $s\neq 1$.  This analysis is undertaken in Section~\ref{sec:gp-new-mod}, where it is shown that the algorithm improves upon previous algorithms for $d\geq 3$, with polynomials of optimal size produced when $d=3$.  However, this improvement is offset in part by the additional complexity of determining suitable parameters $m$, $p$ and $k$.

\section{Length $d+1$ construction revisited}\label{sec:gp-new}

Each of the length $d+1$ GP constructions discussed in Section~\ref{sec:gp-exist} gave rise to geometric progressions $[\seq{c}{0}{d}]$ such that $[\seq{c}{0}{d-1}]$ forms a rational GP.  The following theorem determines all geometric progressions with this property that, in addition, satisfy the properties necessary for polynomial generation:
\begin{theorem}\label{thm:gp-classify}  For $d\geq 2$, $[\seq{c}{0}{d}]$ is a GP modulo $N$ such that
\begin{enumerate}
  \item\label{gp-classify-i} $\seq{c}{0}{d}$ are nonzero,
  \item\label{gp-classify-ii} $\gcd(c_{0},N)=1$,
	\item\label{gp-classify-iii} $[\seq{c}{0}{d-1}]$ is a rational GP, and
	\item\label{gp-classify-iv} $[\seq{c}{0}{d-1},c_{d}]$ is not a rational GP
\end{enumerate}
if and only if there exist nonzero integers $a$, $p$, $m$ and $k$ such that $\gcd(m,p)=1$ and $\gcd(ap,N)=1$, $(am^d-kN)/p$ is a nonzero integer, and
\begin{equation}\label{eqn:gp-new-def}
  [\seq{c}{0}{d}]=\left[ap^{d-1},ap^{d-2}m,\ldots,am^{d-1},\frac{am^{d}-kN}{p}\right].
\end{equation}
\end{theorem}
\begin{proof}Suppose that $d\geq 2$ and $[\seq{c}{0}{d}]$ is a GP modulo $N$ that satisfies properties \eqref{gp-classify-i}--\eqref{gp-classify-iv} of the theorem.  Then properties \eqref{gp-classify-i} and \eqref{gp-classify-iii} imply that there exist nonzero integers $a$, $p$, and $m$ such that $\gcd(m,p)=1$ and $c_{i}=ap^{d-i-1}m^{i}$, for $0\leq i\leq d-1$.  Consequently, $\gcd(ap,N)=1$ as a result of the second property.  As $[\seq{c}{0}{d}]$ is a GP modulo $N$, it follows that $c_{1}c_{d-1}-c_{0}c_{d}\equiv 0\pmod{N}$.  If $c_{1}c_{d-1}-c_{0}c_{d}=0$, then $c_{d}=am^{d}/p$ and $[\seq{c}{0}{d}]$ is a rational GP, violating the fourth property.  Therefore, there exists a nonzero integer $u$ such that
\begin{equation*}
  uN=c_{1}c_{d-1}-c_{0}c_{d}=ap^{d-2}\left(am^{d}-pc_{d}\right).
\end{equation*}
Thus, $am^{d}-pc_{d}=kN$ for some nonzero integer $k$, since $\gcd(ap,N)=1$.  Hence, $(am^{d}-kN)/p$ is a nonzero integer and \eqref{eqn:gp-new-def} holds.  The converse is readily established.
\end{proof}

The GP construction provided by Theorem~\ref{thm:gp-classify} encompasses each of the length $d+1$ constructions discussed in Section~\ref{sec:gp-exist}.  The inclusion of the parameter $a$, which is not present in previous constructions, permits a greater number of geometric progressions to be constructed for any given $N$.  Parameters for the construction may be obtained by the Chinese remainder theorem based methodology of Kleinjung~\cite{kleinjung06}, and Koo, Jo and Kwon~\cite[Section 4.1]{koo11}.  Heuristically, $ax^{d}-kN$ has on average one root modulo each prime, when the average is computed over all primes (see \cite[Chapter~VIII, Section~4]{lang70}).
 Furthermore, roots may be lifted (see \cite[Section~3.5.3]{cohen93}) to higher powers for primes which do not divide
\begin{equation*}
  \res\left(ax^{d}-kN,adx^{d-1}\right)=(ad)^{d}(-kN)^{d-1}.
\end{equation*}
Thus, for any choice of $a$ and $k$, there is an abundant supply of values for $p$ and $m$ that may be used in the construction.

For a GP $\vec{c}$ of the form \eqref{eqn:gp-new-def}, the following lemma provides an efficient method for computing a basis of the lattice $(\vec{c}\Z)^{\bot}$:
\begin{lemma}\label{lem:gp-basis} Let $\vec{c}=[c_{0},\ldots,c_{d}]$ be a GP of the form \eqref{eqn:gp-new-def}, where $d\geq 2$ and $a$, $p$, $m$ and $k$ are nonzero integers such that $\gcd(m,p)=1$, $\gcd(ap,N)=1$ and $(am^d-kN)/p$ is a nonzero integer.  Define $\tilde{a}=a/\gcd(a,c_{d})$ and $\tilde{k}=k/\gcd(a,c_{d})$.  Then for any degree $d$ polynomial $\tilde{f}\in\Z[x]$ with leading coefficient $\tilde{a}$ and $\tilde{f}(m/p)p^{d}=\tilde{k}N$, a vector $(a_{0},\ldots,a_{d})\in\Z^{d+1}$ is orthogonal to $\vec{c}$ if and only if there exist integers $r_{0},\ldots,r_{d-1}$ such that
\begin{equation}\label{eqn:gp-basis}
  \sum^{d}_{i=0}a_{i}x^{i}=r_{d-1}\tilde{f}(x)+(px-m)\cdot\sum^{d-2}_{i=0}r_{i}x^{i}.
\end{equation}
Moreover, such a polynomial $\tilde{f}$ exists and can be computed with Algorithm~\ref{alg:base-m-p}.
\end{lemma}
\begin{proof}  By construction, $am^{d}-kN=c_{d}p$ and $\gcd(a,N)=1$.  Thus, $\tilde{k}$ is an integer and $\tilde{a}m^{d}-\tilde{k}N\equiv 0\pmod{p}$.  Therefore, by modifying arguments of Kleinjung~\cite[Lemma~2.1]{kleinjung06}, it is shown that applying Algorithm~\ref{alg:base-m-p} (which appears below) with $n=j=d$ and $a_{n}=\tilde{a}$ returns a degree $d$ polynomial $\tilde{f}\in\Z[x]$ with leading coefficient $\tilde{a}$ and $\tilde{f}(m/p)p^{d}=\tilde{k}N$.  Fix such a polynomial $\tilde{f}$ and let $(a_{0},\ldots,a_{d})\in\Z^{d+1}$ be orthogonal to $\vec{c}$.  Then
\begin{equation*}
  \sum^{d}_{i=0}a_{i}\left(\frac{m}{p}\right)^{i}-\frac{a_{d}}{\tilde{a}}\tilde{f}\left(\frac{m}{p}\right)%
  =\frac{1}{p^{d}}\left(\frac{p}{a}\sum^{d}_{i=0}a_{i}c_{i}+\frac{a_{d}k}{a}N-\frac{a_{d}\tilde{k}}{\tilde{a}}N\right)=0,
\end{equation*}
where $a_{d}/\tilde{a}$ is an integer, since $a$ divides each term of the sum $\sum^{d-1}_{i=0}a_{i}c_{i}=-a_{d}c_{d}$ and $c_{d}\neq 0$.  Therefore, the assumption that $\gcd(m,p)=1$ and Gauss' lemma imply that there exists a polynomial $r\in\Z[x]$ such that
\begin{equation*}
  \sum^{d}_{i=0}a_{i}x^{i}-\frac{a_{d}}{\tilde{a}}\tilde{f}(x)=r(x)\cdot(px-m).
\end{equation*}
Moreover, since the polynomial on the left-hand side has degree less than $d$, the degree of $r$ is at most $d-2$.  Thus, if a vector $(a_{0},\ldots,a_{d})\in\Z^{d+1}$ is orthogonal to $\vec{c}$, then there exist integers $r_{0},\ldots,r_{d-1}$ such that \eqref{eqn:gp-basis} holds.  Conversely, suppose that $(a_{0},\ldots,a_{d})\in\Z^{d+1}$ and \eqref{eqn:gp-basis} holds for integers $r_{0},\ldots,r_{d-1}$.  Then
\begin{equation*}
  \sum^{d}_{i=0}a_{i}c_{i}=\frac{r_{d-1}}{p}\left(a\tilde{f}\left(\frac{m}{p}\right)p^{d}-\tilde{a}kN\right)+\left(p\frac{m}{p}-m\right)\cdot\sum^{d-2}_{i=0}r_{i}c_{i}=0,
\end{equation*}
since $\tilde{f}(m/p)p^{d}=\tilde{a}{k}N/a$.  Thus, the vector $(a_{0},\ldots,a_{d})$ is orthogonal to $\vec{c}$.
\end{proof}

The following algorithm, which is used to prove the existence of the polynomial $\tilde{f}$ in Lemma~\ref{lem:gp-basis}, is based on the CADO-NFS~\cite{CADO} implementation of a number field sieve polynomial construction of Kleinjung~\cite[Lemma~2.1]{kleinjung06}:
\begin{algorithm}\label{alg:base-m-p}\

{\noindent {\scshape Input:} Nonzero integers $m$, $p$, $\tilde{k}$ and $N$ such that $\gcd(m,p)=1$; and integers $a_{j},\ldots,a_{n}$ such that $1\leq j\leq n$ and $p^{n-j+1}$ divides $\tilde{k}N-\sum^{n}_{i=j}a_{i}m^{i}p^{n-i}$.}

{\noindent {\scshape Output:} An integer polynomial $\tilde{f}=\sum^{n}_{i=0}a_{i}x^{i}$ such that $\tilde{f}(m/p)p^{n}=\tilde{k}N$.}
\begin{list}{\theenumi.}{\usecounter{enumi}  \labelsep=1ex \labelwidth=2ex \leftmargin=12pt \itemindent=0ex}
  \item If $j=n$, set $r_{j}=\tilde{k}N$; otherwise, set
  \begin{equation*}
    r_{j}=\frac{\tilde{k}N-\sum^{n}_{i=j+1}a_{i}m^{i}p^{n-i}}{p^{n-j}}.
  \end{equation*}
  \item For $i=j-1,\ldots,0$, compute
    \begin{equation*}
      r_{i}=\frac{r_{i+1}-a_{i+1}m^{i+1}}{p}\quad\text{and}\quad a_{i}=\frac{r_{i}+t_{i}p}{m^{i}},
    \end{equation*}
   where $t_{i}\in[-m^{i}/2,m^{i}/2)\cap\Z$ such that $t_{i}\equiv-r_{i}/p\pmod{m^{i}}$.
  \item Compute and return the polynomial $\tilde{f}=\sum^{n}_{i=0}a_{i}x^{i}$.
\end{list}
\end{algorithm}

Suppose that $\vec{c}=[c_{0},\ldots,c_{d}]$ is a GP of the form \eqref{eqn:gp-new-def}, and set $\tilde{a}=a/\gcd(a,c_{d})$ and $\tilde{k}=k/\gcd(a,c_{d})$.  Then $\gcd(\tilde{a},\tilde{k})=\gcd(\tilde{k},p)$, since $am^{d}-kN=c_{d}p$ and $\gcd(m,p)=1$.  Thus, if $g=\gcd(\tilde{a},\tilde{k})$ is not equal to $1$, then the vector $\vec{c}^{*}=[c_{0}/g^{d},c_{1}/g^{d-1},\ldots,c_{d}]$ has integer entries, is a GP with ratio $gm/p$ modulo $N$, and has smaller terms than $\vec{c}$:  for all $s>0$,
\begin{equation*}
  \norm{\vec{c}^{*}}_{2,s^{-1}}<\norm{\vec{c}}_{2,s^{-1}}\quad\text{and}\quad\norm{\vec{c}^{*}}_{2,(sg)^{-1}}=g^{-\frac{d}{2}}\norm{\vec{c}}_{2,s^{-1}}.
\end{equation*}
From Lemma~\ref{lem:gp-basis} it follows that a degree $d$ polynomial $f$ generated using $\vec{c}$ has leading coefficient divisible by $\tilde{a}$, and $\tilde{k}N$ divides $f(m/p)p^{d}$.  Thus, the selection of the parameters $a$, $p$ and $k$ may be used to positively influence root properties.  For example, if $a=1$, then Koo, Jo and Kwon~\cite[Remark~5]{koo11} suggest using values of $k$ that contain a product of small primes, ensuring there is a root modulo each prime in the product.  However, care must be taken to ensure that the selection of $a$, $p$ and $k$ does not preclude the existence of polynomial pairs with small coefficients:
\begin{corollary}\label{cor:gp-basis}  With notation as in Lemma~\ref{lem:gp-basis}, if coprime degree $d$ polynomials $f_{1},f_{2}\in\Z[x]$ have coefficient vectors that belong to the lattice $(\vec{c}\Z)^{\bot}$, then
\begin{equation}\label{eqn:cor-gp-basis}
	\big|\tilde{a}\tilde{k}N\big|^{1/d}\leq\left|\sin\theta_{s}(f_{1},f_{2})\right|\cdot\norm{f_{1}}_{2,s}\norm{f_{2}}_{2,s},\quad\text{for all $s>0$}.
\end{equation}	
\end{corollary}
\begin{proof}  Suppose that polynomials $f_{1}$ and $f_{2}$ satisfy the conditions of the corollary.  Then Lemma~\ref{cor:gp-basis} implies that there exist integers $\seq{r}{1,0}{1,d-1},\seq{r}{2,0}{2,d-1}$ and a degree $d$ polynomial $\tilde{f}\in\Z[x]$ with leading coefficient $\tilde{a}$ and $\tilde{f}(m/p)p^{d}=\tilde{k}N$ such that $f_{i}=r_{i,d-1}\tilde{f}+(px-m)\cdot\sum^{d-2}_{j=0}r_{i,j}x^{j}$, for $i=1,2$.  Let $\mathbb{A}=\Z[\seq{y}{0}{d}]$ for algebraically independent indeterminates $\seq{y}{0}{d}$ and define the following polynomials in $\mathbb{A}[x]$: $\tilde{g}=\sum^{d}_{j=0}y_{j}x^{j}$; and $g_{i}=f_{i}+r_{i,d-1}(\tilde{g}-\tilde{f})$, for $i=1,2$.

Let $\varphi:\mathbb{A}\rightarrow\C$ be a ring homomorphism such that $\varphi(\tilde{g}(m/p)p^{d})=0$ and $\tilde{\varphi}:\mathbb{A}[x]\rightarrow\C[x]$ be the natural extension of $\varphi$.  If $\tilde{\varphi}(g_{1})=0$ or $\tilde{\varphi}(g_{2})=0$, then $\varphi(\res(g_{1},g_{2}))=0$, since $\res(g_{1},g_{2})$ is homogeneous of degree $d$ in the coefficients of $g_{i}$, for $i=1,2$.  If $\tilde{\varphi}(g_{1})$ and $\tilde{\varphi}(g_{2})$ are nonzero, then they are non-constant and have common root $m/p$.  Thus, $\res\left(\tilde{\varphi}(g_{1}),\tilde{\varphi}(g_{2})\right)=0$, and by replacing each entry of $\mathrm{Syl}(g_{1},g_{2})$ by its image under $\varphi$ and permuting the rows of the resulting matrix so that $\mathrm{Syl}\left(\tilde{\varphi}(g_{1}),\tilde{\varphi}(g_{2})\right)$ appears in the lower right corner, it follows that $\varphi\left(\res(g_{1},g_{2})\right)=0$.  Therefore, for every ring homomorphism $\varphi:\mathbb{A}\rightarrow\C$, if $\varphi(\tilde{g}(m/p)p^{d})=0$ then $\varphi\left(\res(g_{1},g_{2})\right)=0$.  The polynomial $\tilde{g}(m/p)p^{d}\in\mathbb{A}$ is irreducible, since it is linear and $\gcd(m,p)=1$.  Thus, Hilbert's Nullstellensatz (see \cite[p.~380]{lang02}) and Gauss' lemma imply that $\tilde{g}(m/p)p^{d}$ divides $\res(g_{1},g_{2})$ in $\mathbb{A}$.  Moreover, since $y_{d}$ is coprime to $\tilde{g}(m/p)p^{d}$ and divides each entry in the first column of $\mathrm{Syl}(g_{1},g_{2})$, it follows that $y_{d}\tilde{g}(m/p)p^{d}$ divides $\res(g_{1},g_{2})$ in $\mathbb{A}$.

Let $\sigma:\mathbb{A}\rightarrow\Z$ be the ring homomorphism such that $\tilde{f}=\sum^{d}_{j=0}\sigma(y_{j})x^{j}$ and $\tilde{\sigma}:\mathbb{A}[x]\rightarrow\Z[x]$ be the natural extension of $\sigma$.  Then $\tilde{\sigma}(g_{1})=f_{1}$ and $\tilde{\sigma}(g_{2})=f_{2}$.  Consequently, replacing each entry of $\mathrm{Syl}(g_{1},g_{2})$ by its image under $\sigma$ shows that
\begin{equation*}
  \sigma(\res(g_{1},g_{2}))=\res\left(\tilde{\sigma}(g_{1}),\tilde{\sigma}(g_{2})\right)=\res(f_{1},f_{2}).
\end{equation*}
Moreover, $\sigma(y_{d}\tilde{g}(m/p)p^{d})=\tilde{a}\tilde{f}(m/p)p^{d}$.  Therefore, $\tilde{a}\tilde{k}N$ divides $\res(f_{1},f_{2})$, which is nonzero since $f_{1}$ and $f_{2}$ are coprime.  Hence, $|\tilde{a}\tilde{k}N|\leq|\res(f_{1},f_{2})|$ and the proof is completed by applying Lemma~\ref{lem:skewed-resultant-bnd}.
\end{proof}

Theorem~\ref{thm:gp-classify} and Lemma~\ref{lem:gp-basis} suggest the following algorithm for nonlinear polynomial generation:
\begin{algorithm}\label{alg:nfs-pol-gp}\

{\noindent {\scshape Input:} An integer $d\geq 2$; nonzero integers $a$, $p$, $m$, $k$ and $N$ such that $\gcd(m,p)=1$, $\gcd(ap,N)=1$ and $(am^{d}-kN)/p$ is a nonzero integer; and a positive integer $s$.}

{\noindent {\scshape Output:} Integer polynomials $f_{1}$ and $f_{2}$ of degree at most $d$ with common root $m/p$ modulo $N$.}
\begin{list}{\theenumi.}{\usecounter{enumi}  \labelsep=1ex \labelwidth=2ex \leftmargin=12pt \itemindent=0ex}
  \item\label{step:alg:nfs-pol-gp-i} Compute $g=\gcd(a,(am^{d}-kN)/p)$ and apply Algorithm~\ref{alg:base-m-p} with $n=j=d$, $a_{n}=a/g$ and $\tilde{k}=k/g$ to obtain an integer polynomial $\tilde{f}=\sum^{d}_{i=0}a_{i}x^{i}$ such that $\tilde{f}(m/p)p^{d}=\tilde{k}N$.
  \item\label{step:alg:nfs-pol-gp-iii} Compute an LLL-reduced basis $(\seq{\vec{b}}{1}{d})$ of the lattice with $d\times(d+1)$ basis matrix $BS$, where $S=\mathrm{diag}(1,s,\ldots,s^{d})$ and
     \begin{equation*}
     B=\begin{pmatrix}
        a_{0} & \ldots & \ldots & a_{d-1} &    a_{d} \\
             0 & \ldots &     -m &       p &        0 \\
        \vdots & \revdots & \revdots &  \vdots &   \vdots \\
               -m &      p & \ldots &       0 &        0
     \end{pmatrix}.
    \end{equation*}
  \item\label{step:alg:nfs-pol-gp-iv} Let $\vec{b}_{i}=(b_{i,0},\ldots,b_{i,d})\cdot S$, for $i=1,2$.  Compute and return the polynomials $f_{1}=\sum^{d}_{j=0}b_{1,j}x^{j}$ and $f_{2}=\sum^{d}_{j=0}b_{2,j}x^{j}$.
\end{list}
\end{algorithm}

Parameter selection for Algorithm~\ref{alg:nfs-pol-gp} is considered in the next section.  To guide the selection of parameters, an upper bound on the size of the polynomials returned by the algorithm, expressed as a function of the algorithm's parameters, is now established:
\begin{theorem}\label{thm:nfs-pol-gp} For inputs $d$, $a$, $p$, $m$, $k$, $N$ and $s$, Algorithm~\ref{alg:nfs-pol-gp} returns polynomials $f_{1}$ and $f_{2}$ such that
\begin{equation*}
  \norm{f_{i}}_{2,s}\leq s^{-\frac{\deg f_{i}}{2}}\left(2^{\frac{d(d-1)}{4}}s^{\frac{d^{2}}{2}}\left(\tilde{a}/a\right)\norm{\vec{c}}_{2,s^{-1}}\right)^{\frac{1}{d-i+1}},\quad\text{for $i=1,2$},
\end{equation*}
where $\vec{c}=[c_{0},\ldots,c_{d}]$ is the GP defined in \eqref{eqn:gp-new-def} and $\tilde{a}=a/\gcd(a,c_{d})$.
\end{theorem}
\begin{proof}
Suppose that polynomials $f_{1},f_{2}\in\Z[x]$ are obtained by applying Algorithm~\ref{alg:nfs-pol-gp} with parameters $d$, $a$, $p$, $m$, $k$, $N$ and $s$.  Then $\norm{f_{i}}_{2,s}=s^{-(\deg f_{i})/2}\norm{\vec{b}_{i}}_{2}$ for $i=1,2$, where $\vec{b}_{1}$ and $\vec{b}_{2}$ belong to the basis $(\seq{\vec{b}}{1}{d})$ computed in Step~\ref{step:alg:nfs-pol-gp-iii} of the algorithm.  Let $\vec{c}=[c_{0},\ldots,c_{d}]$ be the GP defined in \eqref{eqn:gp-new-def}, $L$ be the $1$-dimensional lattice $\vec{c}\Z$ and $S=\mathrm{diag}(1,s,\ldots,s^{d})$.  Then Lemma~\ref{lem:gp-basis} implies that $(\seq{\vec{b}}{1}{d})$ is a basis of the lattice $L^{\bot}_{S}$.  Moreover, since $(\seq{\vec{b}}{1}{d})$ is LLL-reduced, it follows from property~\ref{lll:bi-bnd} of Theorem~\ref{thm:LLL-properties} that
\begin{equation}\label{thm:nfs-pol-gp-i}
  \norm{f_{i}}_{2,s}\leq 2^{\frac{d(d-1)}{4(d-i+1)}}s^{-\frac{\deg f_{i}}{2}}\det(L^{\bot}_{S})^{\frac{1}{d-i+1}},\quad\text{for $i=1,2$}.
\end{equation}
As $m$ and $p$ are coprime, $\gcd(c_{0},\ldots,c_{d-1})=a$ and thus $\gcd(c_{0},\ldots,c_{d})=\gcd(a,c_{d})$.  Therefore, Theorem~\ref{thm:det-ortho-lattice} implies that the determinant of $L^{\bot}_{S}$ is
\begin{equation}\label{thm:nfs-pol-gp-ii}
  \det L^{\bot}_{S}=\frac{\tilde{a}}{a}s^{\binom{d+1}{2}}\sqrt{c^{2}_{0}+\frac{c^{2}_{1}}{s^{2}}+\ldots+\frac{c^{2}_{d}}{s^{2d}}}=\frac{\tilde{a}}{a}s^{\frac{d^{2}}{2}}\norm{\vec{c}}_{2,s^{-1}},
\end{equation}
where $\tilde{a}=a/\gcd(a,c_{d})$. Combining \eqref{thm:nfs-pol-gp-i} and \eqref{thm:nfs-pol-gp-ii} completes the proof.
\end{proof}

\subsection{Parameter selection for Algorithm~\ref{alg:nfs-pol-gp}}\label{sec:gp-new-param}

In this section, it is shown that for sufficiently large $N$, Algorithm~\ref{alg:nfs-pol-gp} can be used to find degree $d$ polynomials $f_{1},f_{2}\in\Z[x]$ with a common root modulo $N$ and
\begin{equation}\label{eqn:gp-fi-bnd}
  \norm{f_{i}}_{2,s}=O\left(N^{(1/d)(d^{2}-2d+2)/(d^{2}-d+2)}\right),\quad\text{for $i=1,2$}.
\end{equation}
Thus, polynomials of size $O(N^{1/4})$ are obtained for $d=2$; $O(N^{5/24})$, for $d=3$; and $O(N^{5/28})$, for $d=4$.  Corollary~\ref{cor:skewed-resultant-bnd} implies that the exponent for $d=2$ is optimal.  The bound \eqref{eqn:gp-fi-bnd} is obtained without any assumptions on the size of vectors in LLL-reduced bases, which is in contrast to the previous analyses of \cite{prest10,koo11}.

Theorem~\ref{thm:nfs-pol-gp} suggests that the input parameters $a$, $p$, $m$, $k$ and $s$ of Algorithm~\ref{alg:nfs-pol-gp} should be selected so that $\norm{\vec{c}}_{2,s}$ is minimised, where $\vec{c}=[c_{0},\ldots,c_{d}]$ is the corresponding GP defined in \eqref{eqn:gp-new-def}.  Assume that $a$, $p$ and $k$ are positive, and define $\tilde{m}=(kN/a)^{1/d}$.  Then the contribution of $c_{d}$ to $\norm{\vec{c}}_{2,s^{-1}}$ is minimised by selecting $m$ such that $|m-\tilde{m}|$ is small.  Therefore, assume that $m$ is chosen such that $am^{d}\equiv kN\pmod{p}$ and $0\leq m-\tilde{m}\leq ps/d$.  Then
\begin{equation*}\label{eqn:gp-cd-bnd}
  \frac{c_{d}}{s^{\frac{d}{2}}}=\frac{a(m^{d}-\tilde{m}^{d})}{ps^{\frac{d}{2}}}%
                     <\frac{da(m-\tilde{m})m^{d-1}}{ps^{\frac{d}{2}}}%
                     \leq ap^{d-1}s^{\frac{d}{2}}\left(\frac{m}{ps}\right)^{d-1}.
\end{equation*}
The remaining terms of the GP satisfy
\begin{equation*}\label{eqn:gp-ci-bnd}
  \frac{c_{i}}{s^{i-\frac{d}{2}}}=ap^{d-1}s^{\frac{d}{2}}\left(\frac{m}{ps}\right)^{i},\quad\text{for $0\leq i\leq d-1$}.
\end{equation*}
Therefore, assume that $ps\leq m$.  Then
\begin{equation}\label{eqn:gp-norm-c-bnd}
  \norm{\vec{c}}_{2,s^{-1}}<\sqrt{d+1}\,as^{1-\frac{d}{2}}m^{d-1}.
\end{equation}
This bound on $\norm{\vec{c}}_{2,s^{-1}}$ is minimalised by taking the skew parameter $s$ as large as possible.  However, if $s$ is too large, then the basis $(\vec{b}_{1},\ldots,\vec{b}_{d})$ computed in Step~\ref{step:alg:nfs-pol-gp-iii} of Algorithm~\ref{alg:nfs-pol-gp} has $\vec{b}_{1}=\pm(-m,ps,0,\ldots,0)$, resulting in the algorithm returning $f_{1}=\pm(px-m)$.  Following the approach of Prest and Zimmermann~\cite[Section~3.2]{prest10}, the problem of the algorithm returning polynomials such that $\deg f_{1}<d$ is avoided by restricting the size of $s$.

If Algorithm~\ref{alg:nfs-pol-gp} returns a polynomial $f$ of degree less than $d$, then Lemma~\ref{lem:gp-basis} implies that $m$ divides the coefficient of the term of least degree in $f$, from which it follows that $\norm{f}_{2,s}>s^{-(\deg f)/2}m$, for all $s>0$.  Therefore, Theorem~\ref{thm:nfs-pol-gp} and inequality~\eqref{eqn:gp-norm-c-bnd} imply that for $s$ such that
\begin{equation}\label{eqn:gp-new-s-cond}
  2^{\frac{d-1}{4}}s^{\frac{d}{2}}\left(\sqrt{d+1}\,\tilde{a}s^{1-\frac{d}{2}}m^{d-1}\right)^{\frac{1}{d}}\leq m,
\end{equation}
Algorithm~\ref{alg:nfs-pol-gp} returns polynomials $f_{1}$ and $f_{2}$ such that $\deg f_{1}=d$.  Furthermore, the basis vector $\vec{b}_{1}$ corresponding to $f_{1}$ and the vector $\vec{b}=(-m,sp,0,\ldots,0)$ are linearly independent and $\norm{\vec{b}_{1}}_{2}\leq m<\norm{\vec{b}}_{2}$.  Thus, property \ref{lll:bi-minima-bnd} of Theorem~\ref{thm:LLL-properties} implies that
\begin{equation*}
  \norm{f_{2}}_{2,s}=s^{-\frac{\deg f_{2}}{2}}\norm{\vec{b}_{2}}_{2}\leq s^{-\frac{\deg f_{2}}{2}}2^{\frac{d-1}{2}}\norm{\vec{b}}_{2}\leq s^{-\frac{\deg f_{2}}{2}}2^{\frac{d}{2}}m.
\end{equation*}
Hence, for positive $a$, $p$ and $k$, if $m$ and $s$ are chosen such that
\begin{equation}\label{eqn:param-select}
  0\leq m-\tilde{m}\leq\frac{ps}{d},\quad%
  s=\left\lfloor\frac{1}{\sqrt{2}}\left(\frac{m}{\tilde{a}}\sqrt{\frac{2}{d+1}}\right)^{\frac{2}{d^{2}-d+2}}\right\rfloor\quad\text{and}\quad%
  ps\leq m,
\end{equation}
then Algorithm~\ref{alg:nfs-pol-gp} returns polynomials $f_{1}$ and $f_{2}$ such that $\deg f_{1}=d$ and
\begin{equation*}\label{eqn:gp-new-f1-bnd}
 \norm{f_{i}}_{2,s}=O\left(\tilde{a}^{\frac{\deg f_{i}}{d^{2}-d+2}}\left(\frac{kN}{a}\right)^{\frac{1}{d}\left(1-\frac{\deg f_{i}}{d^{2}-d+2}\right)}\right),\quad\text{for $i=1,2$}.
\end{equation*}
Setting $a=O(1)$ and $k=O(1)$ leads to $f_{1}$ satisfying the bound in \eqref{eqn:gp-fi-bnd}.  Then \eqref{eqn:gp-fi-bnd} holds, if the degree of $f_{2}$ is equal to $d$.  Otherwise, replacing $f_{2}$ with $f_{1}+f_{2}$ provides a pair of degree $d$ polynomials that satisfy \eqref{eqn:gp-fi-bnd}.  In the latter case, it is no longer guaranteed that polynomials with $|\sin\theta_{s}|\geq\sqrt{2/3}$ are obtained.

For $m\geq\tilde{m}$, the choice of $s$ in \eqref{eqn:param-select} satisfies $s/d\geq 1$ whenever
\begin{equation}\label{eqn:m-tilde-lbnd}
  \tilde{m}\geq 2^{\frac{d(d-1)}{4}}a(d+1)^{\frac{d^{2}-d+3}{2}}.
\end{equation}
Moreover, if $m$ is positive and $s$ is given by \eqref{eqn:param-select}, then
\begin{equation*}
  \frac{m}{s}\geq\left(2^{\frac{d(d-1)}{4}}\tilde{a}\sqrt{d+1}\,m^{\frac{d(d-1)}{2}}\right)^{\frac{2}{d^{2}-d+2}}>1.
\end{equation*}
Thus, for sufficiently large $N$, setting $a=k=p=1$ and $m=\left\lceil N^{1/d}\right\rceil$ proves the existence of parameters for Algorithm~\ref{alg:nfs-pol-gp} that satisfy \eqref{eqn:param-select}. 

\begin{example}\label{ex:c91-pz}  Let $N$ be the $91$-digit composite number
\begin{align*}
  \mathrm{c}91
   =&4567176039894108704358752160655628192034927306\\
    &969828397739074346628988327155475222843793393.
\end{align*}
For parameters $d=3$, $a=k=p=1$, $m=\left\lceil N^{1/3}\right\rceil$ and $s=23271635$, Algorithm~\ref{alg:nfs-pol-gp} returns the following cubic polynomials:
\begin{align*}
  f_{1}&=10363104x^{3}
 &f_{2}&=66955475x^{3}\\
       &-23437957x^{2}
       &&-151431419x^{2}\\
       &-21147168576512214234486x
       &&+23469760045042762614639x\\
       &-109084939899748327411476171840
       &&-754597461912921474902918473271
\end{align*}
Their norms are $\norm{f_{1}}_{2,s}\approx N^{0.206}$ and $\norm{f_{2}}_{2,s}\approx N^{0.210}$, for $s=23271635$.
\end{example}

For positive $a$, $p$ and $k$ such that \eqref{eqn:m-tilde-lbnd} holds, if
\begin{equation*}
  p\leq\left(2^{\frac{d(d-1)}{4}}\sqrt{d+1}\,\tilde{m}^{\frac{d(d-1)}{2}}\right)^{\frac{2}{d^{2}-d+2}},
\end{equation*}
then each solution $r\in\Z/p\Z$ of $ax^{d}\equiv kN\pmod{p}$ yields at least one value of $m\in\Z$ such that $m\equiv r\pmod{p}$, $0\leq m-\tilde{m}\leq ps/d$ and $ps\leq m$ for the choice of $s$ in \eqref{eqn:param-select}.  Therefore, additional parameters that satisfy \eqref{eqn:param-select} may be constructed by the Chinese remainder theorem based methodology referenced in Section~\ref{sec:gp-new}.  These parameters yield the same theoretical bounds on coefficient size.  However, the ability to construct a large number of small geometric progressions may be leveraged to find polynomials
with smaller coefficients in practice:
\begin{example}\label{ex:c91} Let $N=\mathrm{c}91$, the $91$-digit number from Example~\ref{ex:c91-pz}.  For parameters $d=3$, $a=1$, $k=5$, $s=26611809$,
\begin{equation*}
  p=934237167355490922\quad\text{and}\quad m=2837086552973239856241381969109,
\end{equation*}
which satisfy \eqref{eqn:param-select}, Algorithm~\ref{alg:nfs-pol-gp} returns the following cubic polynomials:
\begin{align*}
  g_{1}&=21545x^{3}
 &g_{2}&=1356640x^{3}\\
       &+3349054x^{2}
       &&+210882368x^{2}\\
       &-10356871479051937193x
       &&-652118673869097609994x\\
       &+1263295294354066431546642250
       &&-11972068980454909092333428939
\end{align*}
The product $\norm{g_{1}}_{2,s}\cdot\norm{g_{2}}_{2,s}$ is approximately $N^{0.368}$, for $s=26611809$.
\end{example}

The condition \eqref{eqn:gp-new-s-cond} on the skew parameter $s$, used to guarantee that Algorithm~\ref{alg:nfs-pol-gp} returns at least one degree $d$ polynomial, is rather pessimistic in practical circumstances, with the condition being sufficient, but frequently far from necessary.  Consequently, in practice, it may prove worthwhile to use skews that are larger than the value suggested in \eqref{eqn:param-select}.
\begin{example}\label{ex:c91-large-skew}  Let $N=\mathrm{c}91$, the $91$-digit number from Example~\ref{ex:c91-pz}.  For parameters $d=3$, $a=1$, $k=1$, $s=23271635$,
\begin{equation*}
  p=310502797375403107200\quad\text{and}\quad m=1659138281393456348393832527057,
\end{equation*}
which satisfy \eqref{eqn:param-select}, Algorithm~\ref{alg:nfs-pol-gp} returns a pair of cubic polynomials whose product of coefficient norms is approximately $N^{0.396}$ for $s=23271635$, and approximately $N^{0.370}$ for $s=5001852224$.  If, instead, any skew $s\in\left(10^{9.37},10^{9.55}\right)$ is used, the algorithm returns the following cubic polynomials:
\begin{align*}
  h_{1}&=2x^{3}
 &h_{2}&=2x^{3}\\
       &- 46088505322x^{2}
       &&-46088505322x^{2}\\
       &+130858683603618028497x
       &&+441361480979021135697x\\
       &+616682434763766331165127093132
       &&-1042455846629690017228705433925
\end{align*}
The product $\norm{h_{1}}_{2,s}\cdot\norm{h_{2}}_{2,s}$ is approximately $N^{0.347}$, for $s=6425664302$.

For parameters $d=3$, $a=1$, $k=1$, any skew $s\in\left(10^{9.25},10^{9.46}\right)$,
\begin{equation*}
  p=633983687139\quad\text{and}\quad m=1659138281147271980652828686480,
\end{equation*}
Algorithm~\ref{alg:nfs-pol-gp} returns the following cubic polynomials:
\begin{align*}
  k_{1}&=8x^{3}
 &k_{2}&=8x^{3}\\
       &-55x^{2}
       &&-55x^{2}\\
       &+157979116111722504146x
       &&+157979116745706191285x\\
       &+78672185263313067882594467256
       &&-1580466095883958912770234219224
\end{align*}
The product $\norm{k_{1}}_{2,s}\cdot\norm{k_{2}}_{2,s}$ is approximately $N^{0.345}$, for $s=4898436262$.
\end{example}

\section{The Koo--Jo--Kwon length $d+2$ construction revisited}\label{sec:gp-new-mod}

By utilising their length $d+2$ GP construction, Koo, Jo and Kwon obtained an algorithm for producing nonlinear polynomials of degree at most $d$ such that the coefficient of $x^{d-1}$ in each polynomial is equal to zero \cite[Corollary 4]{koo11}.  With the aim of producing polynomials with very large skews, the generation of number field sieve polynomials with this property had previously been considered for linear algorithms by Kleinjung~\cite{kleinjung08}.  In this section, it is shown that larger skews than those obtained in Section~\ref{sec:gp-new-param} are able to be used in the Koo--Jo--Kwon algorithm.  As a result, nonlinear polynomial pairs with smaller coefficient norms are found.

Theorem~\ref{thm:gp-classify} suggests the following extension of the length $d+2$ GP construction of Koo, Jo and Kwon~\cite[Section~4.2]{koo11}:  if $a$, $p$, $k$, and $m$ are nonzero integers such that $\gcd(ap,N)=1$, $\gcd(m,p)=1$ and $(am^{d}-kN)/p^{2}$ is a nonzero integer, then
\begin{equation}\label{eqn:gp-new-def-d-1=0}
   [\seq{c}{0}{d+1}]=\left[ap^{d-1},ap^{d-2}m,\ldots,am^{d-1},\frac{am^{d}-kN}{p},\frac{m(am^{d}-kN)}{p^{2}}\right],
\end{equation}
is a GP with ratio $m/p$ modulo $N$.  The Koo--Jo--Kwon construction then corresponds to the case where $a=1$.  If $[\seq{c}{0}{d+1}]$ is a GP defined by \eqref{eqn:gp-new-def-d-1=0}, then 
\begin{equation*}\label{eqn:gp-d+2-lin}
  m[\seq{c}{0}{d}]-p[\seq{c}{1}{d+1}]=[0,\ldots,0,kN,0].
\end{equation*}
Thus, a vector $(a_{0},\ldots,a_{d})\in\Z^{d+1}$ is orthogonal to $[\seq{c}{0}{d}]$ and $[\seq{c}{1}{d+1}]$ if and only if $a_{d-1}=0$ and $(\seq{a}{0}{d})$ is orthogonal to $(\seq{c}{0}{d-2},0,c_{d})/p\in\Z^{d+1}$.  Consequently, the proof of Lemma~\ref{lem:gp-basis} is readily modified to obtain the following analogous result:
\begin{lemma}\label{lem:gp-basis2} Let $[c_{0},\ldots,c_{d+1}]$ be a GP of the form \eqref{eqn:gp-new-def-d-1=0}, where $d\geq 3$ and $a$, $p$, $m$ and $k$ are nonzero integers such that $\gcd(m,p)=1$, $\gcd(ap,N)=1$ and $(am^d-kN)/p^{2}$ is a nonzero integer.  Define $\tilde{a}=a/\gcd(a,c_{d}/p)$ and $\tilde{k}=k/\gcd(a,c_{d}/p)$.  Then for any degree $d$ integer polynomial $\tilde{f}=\sum^{d}_{i=0}\tilde{a}_{i}x^{i}$ with $\tilde{a}_{d}=\tilde{a}$, $\tilde{a}_{d-1}=0$ and $\tilde{f}(m/p)p^{d}=\tilde{k}N$, a vector $(a_{0},\ldots,a_{d})\in\Z^{d+1}$ is orthogonal to $[\seq{c}{0}{d}]$ and $[\seq{c}{1}{d+1}]$ if and only if there exist integers $r_{0},\ldots,r_{d-2}$ such that
\begin{equation*}
  \sum^{d}_{i=0}a_{i}x^{i}=r_{d-2}\tilde{f}(x)+(px-m)\cdot\sum^{d-3}_{i=0}r_{i}x^{i}.
\end{equation*}
Moreover, such a polynomial $\tilde{f}$ exists and can be computed with Algorithm~\ref{alg:base-m-p}.
\end{lemma}
It follows from Lemma~\ref{lem:gp-basis2} that an analogue of Corollary~\ref{cor:gp-basis} holds:  coprime degree $d$ polynomials $f_{1}$ and $f_{2}$ with coefficients vectors that are orthogonal to $[\seq{c}{0}{d}]$ and $[\seq{c}{1}{d+1}]$ satisfy the bound \eqref{eqn:cor-gp-basis} with the left-hand side of the inequality replaced by $|\tilde{a}^{2}\tilde{k}N|$, where $\tilde{a}=a/\gcd(a,c_{d}/p)$ and $\tilde{k}=k/\gcd(a,c_{d}/p)$.  In addition, Lemma~\ref{lem:gp-basis2} suggests the following algorithm:
\begin{algorithm}\label{alg:nfs-pol-gp-d-1=0}\

{\noindent {\scshape Input:} An integer $d\geq 3$; nonzero integers $a$, $p$, $m$ and $k$ such that $\gcd(m,p)=1$, $\gcd(ap,N)=1$ and $(am^{d}-kN)/p^{2}$ is a nonzero integer; and a positive integer $s$.}

{\noindent {\scshape Output:} Integer polynomials $f_{1}$ and $f_{2}$ of degree at most $d$ with common root $m/p$ modulo $N$.}
\begin{list}{\theenumi.}{\usecounter{enumi}  \labelsep=1ex \labelwidth=2ex \leftmargin=12pt \itemindent=0ex}
  \item Compute $g=\gcd(a,(am^{d}-kN)/p^{2})$ and apply Algorithm~\ref{alg:base-m-p} with $n=d$, $j=d-1$, $a_{n}=a/g$, $a_{n-1}=0$ and $\tilde{k}=k/g$ to obtain an integer polynomial $\tilde{f}=\sum^{d}_{i=0}a_{i}x^{i}$ such that $\tilde{f}(m/p)p^{d}=\tilde{k}N$.
  \item Compute an LLL-reduced basis $(\seq{\vec{b}}{1}{d-1})$ of the lattice with $(d-1)\times d$ basis matrix $BS$, where $S=\mathrm{diag}\left(1,s,\ldots,s^{d-2},s^{d}\right)$ and
     \begin{equation*}
     B=\begin{pmatrix}
        a_{0} & \ldots & \ldots & a_{d-2} &    a_{d} \\
             0 & \ldots &     -m &       p &        0 \\
        \vdots & \revdots & \revdots &  \vdots &   \vdots \\
               -m &      p & \ldots &       0 &        0
     \end{pmatrix}.
    \end{equation*}
  \item Let $\vec{b}_{i}=\left(b_{i,0},\ldots,b_{i,d-1}\right)\cdot S$, for $i=1,2$.  Compute and return the polynomials $f_{1}=b_{1,d-1}x^{d}+\sum^{d-2}_{j=0}b_{1,j}x^{j}$ and $f_{2}=b_{2,d-1}x^{d}+\sum^{d-2}_{j=0}b_{2,j}x^{j}$.
\end{list}
\end{algorithm}
In the next section, parameter selection for Algorithm~\ref{alg:nfs-pol-gp-d-1=0} is considered.

Koo, Jo and Kwon~\cite[Section~3]{koo11} observe that a single GP may be used to generate several polynomials of various degrees that share a common root modulo $N$.  In particular, for a GP $\vec{c}$ of the form \eqref{eqn:gp-new-def-d-1=0}, it is possible to generate degree $d$ and $d+1$ polynomials.  In the former case, Algorithm~\ref{alg:nfs-pol-gp} or Algorithm~\ref{alg:nfs-pol-gp-d-1=0} may be used.  In the latter case, the necessary computation of a basis for $(\vec{c}\Z)^{\bot}$ is aided by the following analogue of Lemma~\ref{lem:gp-basis}:
\begin{lemma}\label{lem:gp-basis3} Let $\vec{c}=[c_{0},\ldots,c_{d+1}]$ be a GP of the form \eqref{eqn:gp-new-def-d-1=0}, where $d\geq 1$ and $a$, $p$, $m$ and $k$ are nonzero integers such that $\gcd(m,p)=1$, $\gcd(ap,N)=1$ and $(am^d-kN)/p^{2}$ is a nonzero integer.  Define $\tilde{a}=a/\gcd(a,c_{d}/p)$ and $\tilde{k}=k/\gcd(a,c_{d}/p)$.  Then for any degree $d+1$ integer polynomial $\tilde{f}=\sum^{d}_{i=0}\tilde{a}_{i}x^{i}$ with $\tilde{a}_{d+1}m+\tilde{a}_{d}p=\tilde{a}$ and $\tilde{f}(m/p)p^{d+1}=\tilde{k}N$, a vector $(a_{0},\ldots,a_{d+1})\in\Z^{d+2}$ is orthogonal to $\vec{c}$ if and only if there exist integers $r_{0},\ldots,r_{d}$ such that
\begin{equation}\label{eqn:gp-basis2}
  \sum^{d+1}_{i=0}a_{i}x^{i}=r_{d}\tilde{f}(x)+(px-m)\cdot\left(r_{d-1}x^{d}+\sum^{d-2}_{i=0}r_{i}x^{i}\right).
\end{equation}
Moreover, such a polynomial $\tilde{f}$ exists and can be computed with Algorithm~\ref{alg:base-m-p}.
\end{lemma}
To establish the existence of $\tilde{f}$ in Lemma~\ref{lem:gp-basis3}, observe that integers $\tilde{a}_{d}$ and $\tilde{a}_{d+1}$ such that $\tilde{a}_{d+1}\neq 0$ and $\tilde{a}_{d+1}m+\tilde{a}_{d}p=\tilde{a}$ exist, since $\gcd(m,p)=1$.  For any such $\tilde{a}_{d}$ and $\tilde{a}_{d+1}$, a polynomial $\tilde{f}$ that satisfies the conditions of the lemma is obtained by applying Algorithm~\ref{alg:base-m-p} with $n=d+1$, $j=d$, $a_{n}=\tilde{a}_{d+1}$ and $a_{n-1}=\tilde{a}_{d}$:
\begin{equation*}
  \tilde{k}N-\tilde{a}_{d+1}m^{d+1}-\tilde{a}_{d}m^{d}p\equiv\tilde{k}N-\tilde{a}m^{d}\equiv 0\pmod{p^{2}},
\end{equation*}
since $am^{d}-kN=(c_{d}/p)p^{2}$ and $\gcd(a,N)=1$.  Additionally, setting $r_{d}=p$ and $r_{d-1}=-\tilde{a}_{d+1}$ in \eqref{eqn:gp-basis2} proves the existence of a degree $d$ polynomial $f\in\Z[x]$ with coefficient vector that is orthogonal to $[c_{0},\ldots,c_{d}]$ and $f(m/p)\neq 0$.  Consequently, to avoid finding polynomials of degree less than $d+1$, it is not sufficient to avoid finding polynomials that have $m/p$ as a root in a manner analogous to Section~\ref{sec:gp-new-param}.

\subsection{Parameter selection for Algorithm~\ref{alg:nfs-pol-gp-d-1=0}}\label{sec:gp-new-mod-param}

By modifying the proof of Theorem~\ref{thm:nfs-pol-gp}, it is shown that for input parameters $a$, $p$, $m$, $k$ and $s$, Algorithm~\ref{alg:nfs-pol-gp-d-1=0} returns polynomials $f_{1}$ and $f_{2}$ such that
\begin{equation*}
  \norm{f_{1}}_{2,s}\leq s^{-\frac{\deg f_{1}}{2}}2^{\frac{d-2}{4}}\left(s^{\frac{d^{2}-2d+2}{2}}\frac{\tilde{a}}{a|p|}\sqrt{\frac{c^{2}_{0}}{s^{-d}}+\frac{c^{2}_{1}}{s^{2-d}}+\ldots+\frac{c^{2}_{d-2}}{s^{d-4}}+\frac{c^{2}_{d}}{s^{d}}}\,\right)^{\frac{1}{d-1}},
\end{equation*}
where $[c_{0},\ldots,c_{d+1}]$ is the GP defined in \eqref{eqn:gp-new-def-d-1=0} and $\tilde{a}=a/\gcd(a,c_{d}/p)$.  Furthermore, if $\deg f_{1}<d$, then Lemma~\ref{lem:gp-basis2} implies that $\norm{f_{1}}_{2,s}>s^{-(\deg f_{1})/2}m$, for all $s>0$.  Therefore, repeating arguments from Section~\ref{sec:gp-new-param} shows that for positive $a$, $p$ and $k$, if $m$ and $s$ are chosen such that
\begin{equation}\label{eqn:param-select-d+2}
  0\leq m-\left(\frac{kN}{a}\right)^{\frac{1}{d}}\leq \frac{ps}{d},\quad%
  s=\left\lfloor\frac{1}{\sqrt{2}}\left(\frac{p}{\tilde{a}}\sqrt{\frac{2}{d}}\right)^{\frac{2}{d^{2}-3d+4}}\right\rfloor\quad\text{and}\quad%
  ps\leq m,
\end{equation}
then Algorithm~\ref{alg:nfs-pol-gp-d-1=0} returns polynomials $f_{1}$ and $f_{2}$ such that $\deg f_{1}=d$ and
\begin{equation*}
 \norm{f_{i}}_{2,s}=O\left(\tilde{a}^{\frac{\deg f_{i}}{d^{2}-3d+4}}p^{-\frac{\deg f_{i}}{d^{2}+3d+4}}\left(\frac{kN}{a}\right)^{\frac{1}{d}}\right),\quad\text{for $i=1,2$}.
\end{equation*}
This bound is minimised and the skew parameter maximised by selecting the parameter $p$ as large as possible.  In particular, if the constraints in \eqref{eqn:param-select-d+2} are satisfied and $p=\Theta(m/s)$, then $s=\Theta((m/\tilde{a})^{2/(d^{2}-3d+6)})$, which is larger than the skew in \eqref{eqn:param-select}, and Algorithm~\ref{alg:nfs-pol-gp-d-1=0} returns polynomials $f_{1}$ and $f_{2}$ such that $\deg f_{1}=d$ and
\begin{equation}\label{eqn:f1-bnd-d+2-b}
 \norm{f_{i}}_{2,s}=O\left(\tilde{a}^{\frac{\deg f_{i}}{d^{2}-3d+6}}\left(\frac{kN}{a}\right)^{\frac{1}{d}\left(1-\frac{\deg f_{i}}{d^{2}-3d+6}\right)}\right),\quad\text{for $i=1,2$}.
\end{equation}
Similar to Section~\ref{sec:gp-new-param}, if the degree of $f_{2}$ is not equal to $d$, then a second degree $d$ polynomial that satisfies the bound \eqref{eqn:f1-bnd-d+2-b} is found by taking a linear combination of the polynomials $f_{1}$ and $f_{2}$.  Therefore, if $a=O(1)$ and $k=O(1)$, then it is possible to find a pair of degree $d$ polynomials $f_{1}$ and $f_{2}$ such that
\begin{equation*}\label{eqn:gp-fi-bnd-d+2}
  \norm{f_{i}}_{2,s}=O\left(N^{(1/d)(d^{2}-4d+6)/(d^{2}-3d+6)}\right),\quad\text{for $i=1,2$}.
\end{equation*}
Thus, polynomials of size $O(N^{1/6})$ are obtained for $d=3$, and size $O(N^{3/20})$ for $d=4$.  Corollary~\ref{cor:skewed-resultant-bnd} implies that the exponent for $d=3$ is optimal.

In Section~\ref{sec:gp-new-param}, the observation that each solution $r\in\Z/p\Z$ of $ax^{d}\equiv kN\pmod{p}$ yields a value $m\equiv r\pmod{p}$ such that $0\leq m-\tilde{m}<p$, where $\tilde{m}=(kN/a)^{1/d}$, was used to assert the existence of many parameters that satisfy the conditions leading to the bound \eqref{eqn:gp-fi-bnd}.  However, for parameters that satisfy \eqref{eqn:param-select-d+2}, $ps<p^{1+2/(d^{2}-3d+4)}$.  In particular, $ps$ is less than $p^{3/2}$ for $d=3$, and less than $p^{5/4}$ for $d=4$.  Therefore, it may occur that $ax^{d}\equiv kN\pmod{p^{2}}$ has a solution $r\in\Z/p^{2}\Z$, but there is no integer $m\equiv r\pmod{p^{2}}$ such that $0\leq m-\tilde{m}\leq ps/d$, as required by \eqref{eqn:param-select-d+2}.  This problem is compounded by the requirement that $p$ is taken as large as possible in order to obtain \eqref{eqn:f1-bnd-d+2-b}.  Thus, finding parameters for Algorithm~\ref{alg:nfs-pol-gp-d-1=0} that satisfy the conditions leading to \eqref{eqn:f1-bnd-d+2-b} is difficult.  This problem is addressed in the next section, where methods for generating parameters for Algorithm~\ref{alg:nfs-pol-gp-d-1=0} are discussed.

\subsection{Parameter generation for Algorithm~\ref{alg:nfs-pol-gp-d-1=0}}
Motivated by the discussion of the previous section, methods of generating parameters for Algorithm~\ref{alg:nfs-pol-gp-d-1=0} such that $p\in[B,2B]$ and $|m-\tilde{m}|=O(p^{1+\varepsilon})$, where $B$ is large, $\tilde{m}$ is a real number and $0\leq\varepsilon<1$, are discussed in this section.

Koo, Jo and Kwon~\cite[Proposition~2]{koo11} propose using Hensel lifting for prime values of $p$ to compute corresponding values of $m$ such that $|m-\tilde{m}|$ is small.  For example, after selecting $a$, $p$ and $k$ such that $p\in[B,2B]$ is prime and does not divide $adkN$, suppose there exists a solution $r\in[-p/2,p/2)\cap\Z$ of
\begin{equation}\label{eqn:m0tilde-modp}
 a(\tilde{m}_{0}+x)^{d}\equiv kN\pmod{p},
\end{equation}
where $\tilde{m}_{0}=\left\lfloor\tilde{m}+1/2\right\rfloor$, and there exists an integer $t=O(p^{\varepsilon})$ such that
\begin{equation*}
  ad(\tilde{m}_{0}+r)^{d-1}t\equiv-\frac{a(\tilde{m}_{0}+r)^{d}-kN}{p}\pmod{p}.
\end{equation*}
Then $r^{*}=r+tp$ satisfies $r^{*}=O(p^{1+\varepsilon})$ and Hensel's lemma implies that 
\begin{equation}\label{eqn:m0tilde-modp2}
  a\left(\tilde{m}_{0}+r^{*}\right)^{d}\equiv kN\pmod{p^{2}}.
\end{equation}
Therefore, parameters $a$, $k$, $p$ and $m=\tilde{m}_{0}+r^{*}$ may be used in Algorithm~\ref{alg:nfs-pol-gp-d-1=0} and $|m-\tilde{m}|=O(p^{1+\varepsilon})$.  Arguments used by Bai~\cite[Section~4.2]{bai11} are now adapted to provide heuristic evidence of the existence of solutions $r^{*}$ of \eqref{eqn:m0tilde-modp2} with $r^{*}=O(p^{1+\varepsilon})$.

Heuristically, when $a$ and $k$ are fixed, \eqref{eqn:m0tilde-modp} has on average about one solution for each prime $p\in[B,2B]$.  Thus, for fixed $a$ and $k$, assuming the solutions lift to solutions $r^{*}$ of \eqref{eqn:m0tilde-modp2} that are uniformly distributed in the interval $[-2B^2,2B^{2})$, it follows that the total number of pairs $(p,r^{*})$ with $p\in[B,2B]$ and $r^{*}=O(p^{1+\varepsilon})$ that satisfy \eqref{eqn:m0tilde-modp2} is about
\begin{equation*}
  \sum_{\substack{p\in[B,2B]\\\text{$p$ prime}}}\frac{2B^{1+\varepsilon}}{p^{2}}%
  \approx 2B^{1+\varepsilon}\int^{2B}_{B}\frac{1}{x^{2}\log x}\,dx%
  \approx 2B^{1+\varepsilon}\int^{2B}_{B}\frac{1+\log x}{\left(x\log x\right)^{2}}\,dx%
  \approx\frac{B^{\varepsilon}}{\log B}.
\end{equation*}
Therefore, it is expected that solutions $r^{*}$ of \eqref{eqn:m0tilde-modp2} with $r^{*}=O(p^{1+\varepsilon})$ exist for large $B$.  However, to find such a solution it may be necessary to compute and lift the solutions of \eqref{eqn:m0tilde-modp} for many values of $a$, $k$ and $p$.  This heuristic evidence agrees with numerical evidence provided by Koo, Jo and Kwon~\cite[Section~4.2]{koo11}, which suggests that good parameters for Algorithm~\ref{alg:nfs-pol-gp-d-1=0} with $p$ prime exist with high probability for each $N$, albeit in very small quantities.

A method proposed by Kleinjung~\cite{kleinjung08} in the setting of linear algorithms, but which may be used to generate parameters for Algorithm~\ref{alg:nfs-pol-gp-d-1=0}, reduces the size of the primes $p$ for which \eqref{eqn:m0tilde-modp2} is solved and relaxes the size requirements on the solutions by exploiting a special case of the Chinese remainder theorem.  Suppose that, for fixed values of $a$ and $k$, a solution $r^{*}$ of \eqref{eqn:m0tilde-modp2} is found for two distinct prime values of $p$, say $p_{1}$ and $p_{2}$.  Then
\begin{equation*}
  a\left(\tilde{m}_{0}+r^{*}\right)^{d}\equiv kN\pmod{p^{2}_{1}p^{2}_{2}}.
\end{equation*}
Therefore, parameters $a$, $k$, $p=p_{1}p_{2}$ and $m=\tilde{m}_{0}+r^{*}$ may be used in Algorithm~\ref{alg:nfs-pol-gp-d-1=0}.  Such a value of $r^{*}$ is called a \emph{collision}.  Searching for collisions offers the advantage that the two primes $p_{1}$ and $p_{2}$ need only belong to the interval $[\sqrt{B},\sqrt{2B}]$, reducing the time spent computing solutions of \eqref{eqn:m0tilde-modp2} for each $p_{i}$.  Moreover, if $r^{*}=O(B^{1+\varepsilon})$, then $|m-\tilde{m}|=O(p^{1+\varepsilon})$.  Thus, for each $p_{i}$, solutions of \eqref{eqn:m0tilde-modp2} with $r^{*}=O(p^{2}_{i})$ are computed, which exist with high probability.  However, collisions occur with low probability for large $B$, balancing these benefits.  For an analysis of the method and a description of some modifications aimed at improving efficiency, see Kleinjung~\cite{kleinjung08} and Bai~\cite[Section~4.2]{bai11}.

A third method of parameter generation for Algorithm~\ref{alg:nfs-pol-gp-d-1=0}, which appears to have limited practical value, once again uses the Chinese remainder theorem to construct parameters for Algorithm~\ref{alg:nfs-pol-gp-d-1=0}, but eliminates the need to find collisions.  Suppose that distinct primes $p_{1},\ldots,p_{n}$, integers $r^{*}_{1},\ldots,r^{*}_{n}$, and positive integers $e_{1},\ldots,e_{n}$ are found such that
\begin{equation*}
  a(r^{*}_{i}+\tilde{m}_{0})^{d}\equiv kN\pmod{p^{2e_{i}}_{i}},\quad\text{for $1\leq i\leq n$}.
\end{equation*}
Then, for positive integers $M$, $z$ and $\ell$, existing list decoding algorithms for Chinese remainder codes~\cite{guruswami00,boneh02,coxon13} are able to find all $r^{*}\in[-M,M]\cap\Z$ such that
\begin{equation}\label{eqn:list-crt-bnd}
  \prod^{n}_{i=1}p^{\chi_{i}(r^{*})}_{i}>\left(2^{\frac{\ell}{4}}\sqrt{\ell+1}M^{\frac{\ell}{2}}\right)^{\frac{1}{2z}}\left(\prod^{n}_{i=1}p^{e_{i}}_{i}\right)^{\frac{z+1}{2(\ell+1)}},
\end{equation}
where $\chi_{i}:\Z\rightarrow\{0,\ldots,e_{i}\}$ maps $r^{*}$ to the maximum value $\chi\in\{0,\ldots,e_{i}\}$ such that $r^{*}\equiv r^{*}_{i}\pmod{p^{2\chi}_{i}}$.  For any such $r^{*}$, parameters $a$, $k$, $p=\prod^{n}_{i=1}p^{\chi_{i}(r^{*})}_{i}$ and $m=\tilde{m}_{0}+r^{*}$ may be used in Algorithm~\ref{alg:nfs-pol-gp-d-1=0} and $|m-\tilde{m}|=O(M)$.  Thus, $M\approx B^{1+\varepsilon}$ should be used and the remaining parameters chosen such that the right-hand side of \eqref{eqn:list-crt-bnd} approximates $B$.  This method of parameter generation has not been considered previously in the literature, but offers the advantages that the primes $p_{i}$ may be taken much smaller than in previous methods and that there are no size requirements on the roots $r^{*}_{i}$.  However, list decoding algorithms are computationally expensive:  to compute all solutions of \eqref{eqn:list-crt-bnd}, existing algorithms perform lattice reduction on an $(\ell+1)$-dimensional lattice and compute the integer roots of an integer polynomial of degree at most $\ell$, with arithmetic operations for both steps performed on integers of size polynomial in $\ell$, $z$, $\log M$ and $\sum^{n}_{i=1}e_{i}\log p_{i}$.  Consequently, a search for parameters for Algorithm~\ref{alg:nfs-pol-gp-d-1=0} should avoid the extensive application of current list decoding algorithms.
 
\section*{Acknowledgements}

The author would like to thank Dr Victor Scharaschkin for many helpful discussions and suggestions throughout the preparation of this paper.

\bibliographystyle{amsplain}
\providecommand{\bysame}{\leavevmode\hbox to3em{\hrulefill}\thinspace}

\end{document}